\documentclass[seceqn,secthm,9pts]{article}%
\usepackage{amsfonts}
\usepackage{mitpress}
\usepackage{amsmath}
\usepackage{amssymb}
\usepackage{graphicx}
\providecommand{\U}[1]{\protect\rule{.1in}{.1in}}
\newtheorem{theorem}{Theorem}

\newtheorem{corollary}[theorem]{Corollary}

\newtheorem{definition}[theorem]{Definition}

\newtheorem{remark}[theorem]{Remark}

\newenvironment{proof}[1][Proof]{\noindent \textbf{#1.} }{\  \rule{0.5em}{0.5em}}
\newdimen \dummy
\dummy=\oddsidemargin
\begin{document}

\title{Finite Bivariate Biorthogonal I - Konhauser Polynomials}
\author{Güldoğan Lekesiz, E., Çekim, B.* and Özarslan, M.A.}
\maketitle

\author{Esra G\"{U}LDO\u{G}AN LEKES\.{I}Z\\Çankaya University, Faculty of Arts and Sciences, Department of Mathematics, Ankara 06790, T\"{u}rkiye
\\
esraguldoganlekesiz@cankaya.edu.tr}\\

\author{Bayram \c{C}EK\.{I}M*\\Gazi University, Faculty of Science, Department of Mathematics, Ankara 06500, T\"{u}rkiye
\\
bayramcekim@gazi.edu.tr}\\

\author{Mehmet Ali \"{O}ZARSLAN\\Eastern Mediterranean University, Faculty of Arts and Sciences, Department of Mathematics, Gazimagusa, TRNC, Mersin 10, T\"{u}rkiye
\\
mehmetali.ozarslan@emu.edu.tr}
\\

\begin{abstract}
In the present study, a finite set of biorthogonal polynomials in two
variables, produced from Konhauser polynomials, is introduced. Some properties
like Laplace transform, integral and operational representation, fractional
calculus operators of this family are investigated. Also, we compute Fourier
transform for this new set and discover a new family of finite biorthogonal
functions with the help of Parseval's identity. Further, in order to have
semigroup property, we modify this finite set by adding two new parameters and
construct fractional calculus operators. Thus, integral equation and integral
operator are obtained for the modified version.

\end{abstract}


\section{Introduction}

Orthogonal polynomials have been the subject of different problems occuring in theoretical and applied sciences. Their extension like multivariable generalization, biorthogonal and $d$-orthogonal analogues, matrix forms and $q$-extension have also been intriguid for many years.
For example, \cite{MST} has introduced $q$-extension of Sheffer sequences of Bessel type using $q$-Riordan matrices by means of the Eulerian generating function.
The study \cite{SM} is first attempt in the direction of deriving integral equations for the Appell and 2-iterated Appell families and for some members belonging to these families.
In \cite{SSAM}, the generalized forms of the Legendre polynomials and Legendre-Appell polynomials are introduced and studied by means of fractional operators and the combination of the operational definitions and integral representations.
In \cite{RTMNA}, the Gould-Hopper matrix polynomials combined with the Bessel functions and the Tricomi functions respectively, to construct the new class of hybrid matrix functions. The generalized forms of the Gould-Hopper-Bessel matrix and Gould-Hopper-Tricomi matrix functions are introduced using integral transform and operational rule.

In this study a two-variable biorthogonal extension of finite orthogonal polynomials is studied. For this purpose, we first give the definitions of biorthogonal polynomials and finite orthogonal polynomials.

The biorthogonal polynomials are introduced by Didon \cite{Didon} and Deruyts
\cite{Deruyts}. They are $k$-th degree polynomial solutions of the
differential equation%
\[
B_{1}(t)y^{\prime\prime\prime}+B_{2}(t)y^{\prime\prime}+B_{3}(t)y^{\prime
}=\lambda_{s}y.
\]

\begin{definition}
If%
\[
J_{s,k}=\int\limits_{\alpha_{1}}^{\alpha_{2}}D_{s}(t)F_{k}(t)\rho\left(
t\right)  dt=%
\genfrac{\{}{.}{0pt}{}{\ \ \ 0;\ \ s\neq k}{\neq0;\ \ s=k,}%
,\ s,k\in%
\mathbb{N}
_{0},
\]
the polynomials $D_{s}(t)$ and $F_{k}(t)$ corresponding to the fundamental
polynomials $d(t)$ and $f(t)$\ are biorthogonal with respect to the weight
function $\rho(t)$ on $\left(  \alpha_{1},\alpha_{2}\right)  $
\cite{Konhauser2}.
\end{definition}

We can give the theorem below, equivalently.

\begin{theorem}
Let $d(t)$ and $f(t)$ be fundamental polynomials corresponding to polynomials
$D_{s}(t)$ and$\ F_{k}(t)$, respectively, and $\rho(t)$ be the weight function
over the interval $\left(  \alpha_{1},\alpha_{2}\right)  $. For $s,k\in%
\mathbb{N}
_{0}$, the necessary and sufficient condition is%
\[
J_{s,k}=\int\limits_{\alpha_{1}}^{\alpha_{2}}D_{s}(t)F_{k}(t)\rho\left(
t\right)  dt=%
\genfrac{\{}{.}{0pt}{}{\ \ \ 0;\ \ s\neq k}{\neq0;\ \ s=k}%
\]
so that%
\[
\int\limits_{\alpha_{1}}^{\alpha_{2}}\left[  d\left(  t\right)  \right]
^{j}F_{k}(t)\rho\left(  t\right)  dt=%
\genfrac{\{}{.}{0pt}{}{\ \ 0;\ j=0,1,...,k-1}{\neq
0;\ j=k\ \ \ \ \ \ \ \ \ \ \ \ \ \ \ }%
\]
and%
\[
\int\limits_{\alpha_{1}}^{\alpha_{2}}\left[  f\left(  t\right)  \right]
^{j}D_{s}(t)\rho\left(  t\right)  dt=%
\genfrac{\{}{.}{0pt}{}{\ \ 0;\ j=0,1,...,s-1}{\neq
0;\ j=s\ \ \ \ \ \ \ \ \ \ \ \ \ \ \ }%
,
\]
are provided \cite{Konhauser2}.
\end{theorem}

There are univariate analogues of biorthogonal polynomials in the literature
and many researchers have been studied on these polynomials until today
\cite{Konhauser2, Konhauser, Carlitz, MT, MT2, AlSalamVerma, MT0}. Konhauser
\cite{Konhauser} derived the biorthogonal pairs of polynomials suggested
by\ the Laguerre polynomials corresponding to weight function $t^{\alpha
}e^{-t}$. Toscano \cite{Toscano} and Carlitz \cite{Carlitz} studied the
Konhauser polynomials. In 1982, Madhekar\&Thakare \cite{MT} introduced the
biorthogonal polynomials suggested by the Jacobi polynomials with the weight
function $(1-t)^{\alpha}(1+t)^{\beta}$. Later, in 1986 and 1988, the
biorthogonal pairs corresponding to the Szeg\"{o}-Hermite weight function
$\left\vert t\right\vert ^{2\mu}exp(-t^{2}),\ \mu>-1/2$ and the Hermite weight
function $\exp\left(  -t^{2}\right)  $\ were taken up by Thakare\&Madhekar
\cite{MT2,MT0}. For the biortogonal polynomials, the readers can refer to the references in \cite{Cesarano,BGS,BEH}.

For $\gamma>-1$ and $\upsilon=1,2,...$, one of the well-known biorthogonal
polynomials in one variable is Konhauser polynomials \cite{Konhauser} defined
by%
\begin{equation}
Z_{k}^{\gamma}\left(  t;\upsilon\right)  =\frac{\Gamma\left(  \gamma
+1+\upsilon k\right)  }{k!}\sum\limits_{j=0}^{k}\binom{k}{j}\frac{\left(
-1\right)  ^{j}t^{\upsilon j}}{\Gamma\left(  \upsilon j+\gamma+1\right)  }
\label{Zdef}%
\end{equation}
and%
\begin{equation}
Y_{k}^{\gamma}\left(  t;\upsilon\right)  =\frac{1}{k!}\sum\limits_{j=0}%
^{k}\frac{t^{j}}{j!}\sum\limits_{m=0}^{j}\left(  -1\right)  ^{m}\binom{j}%
{m}\left(  \frac{1+\gamma+m}{\upsilon}\right)  _{k}. \label{Ydef}%
\end{equation}

From \cite{Konhauser}, for the pair (\ref{Zdef}) and (\ref{Ydef}), the corresponding biorthogonality
relation is%
\begin{equation}
\int\limits_{0}^{\infty}Z_{k}^{\gamma}\left(  t;\upsilon\right)  Y_{s}%
^{\gamma}\left(  t;\upsilon\right)  e^{-t}t^{\gamma}dt=\frac{\Gamma\left(
\gamma+\upsilon k+1\right)  }{k!}\delta_{k,s}. \label{3}%
\end{equation}

Recently, some more general biorthogonal polynomials in two variables are also
studied. First, Bin-Saad introduced the Laguerre-Konhauser polynomials
\cite{BinSaad} as follows%
\begin{align}
\ _{\upsilon}L_{k}^{\left(  \mu,\gamma\right)  }\left(  z,t\right)   &
=k!\sum\limits_{l=0}^{k}\sum\limits_{n=0}^{k-l}\frac{\left(  -1\right)
^{n+l}z^{n+\mu}t^{\upsilon l+\gamma}}{\left(  k-l-n\right)  !n!l!\Gamma\left(
n+1+\mu\right)  \Gamma\left(  \upsilon l+1+\gamma\right)  }\label{BS}\\
&  =k!\sum\limits_{l=0}^{k}\frac{\left(  -1\right)  ^{l}z^{l+\mu}t^{\gamma
}Z_{k-l}^{\gamma}\left(  t;\upsilon\right)  }{l!\Gamma\left(  \upsilon
k-\upsilon l+\gamma+1\right)  \Gamma\left(  l+\mu+1\right)  }\nonumber\\
&  =k!\sum\limits_{l=0}^{k}\frac{\left(  -1\right)  ^{l}z^{\mu}t^{\upsilon
l+\gamma}L_{k-l}^{\mu}\left(  z\right)  }{l!\Gamma\left(  \upsilon
l+\gamma+1\right)  \Gamma\left(  k-l+\mu+1\right)  },\nonumber
\end{align}
where $Z_{k}^{\gamma}\left(  t;\upsilon\right)  $ are the first set of
Konhauser polynomials (\ref{Zdef}) and $L_{k}^{\mu}\left(  t\right)  $ are the
generalized Laguerre polynomials, given in \cite{R}, for $\upsilon=1,2,...$ and $\mu,\gamma>-1.$

Then, \"{O}zarslan and K\"{u}rt \cite{OzKurt} derived the second
set$\ _{\upsilon}%
\mathcal{L}%
_{k}^{\left(  \mu,\gamma\right)  }\left(  z,t\right)  $ of the form%
\begin{equation}
\ _{\upsilon}%
\mathcal{L}%
_{k}^{\left(  \mu,\gamma\right)  }\left(  z,t\right)  =L_{k}^{\mu}\left(
z\right)  \sum\limits_{l=0}^{k}Y_{l}^{\gamma}\left(  t;\upsilon\right)  ,
\label{OK}%
\end{equation}
and they showed that polynomials (\ref{OK}) and (\ref{BS}) are biorthonormal
with respect to the weight function $\rho\left(  z,t\right)  =e^{-\left(
t+z\right)  }$ over $\left(  0,\infty\right)  \times\left(  0,\infty\right)
$. Also, they presented a bivariate Mittag-Leffler functions $E_{\mu
,\gamma,\upsilon}^{(\alpha)}\left(  z,t\right)  $, corresponding to the
polynomials $_{\upsilon}L_{k}^{\left(  \mu,\gamma\right)  }\left(  z,t\right)
$, with%
\begin{align}
E_{\mu,\gamma,\upsilon}^{(\alpha)}\left(  z,t\right)  =\sum\limits_{l=0}%
^{\infty}\sum\limits_{n=0}^{\infty}\frac{\left(  \alpha\right)  _{n+l}%
\ z^{l}t^{\upsilon n}}{n!l!\Gamma\left(  \mu+l\right)  \Gamma\left(
\gamma+\upsilon n\right)  } \label{OKEdef},
\end{align}
where $\operatorname{Re}\left(  \upsilon\right)  >0$, $\operatorname{Re}%
\left(  \gamma\right)  >0$, $\operatorname{Re}\left(  \mu\right)  >0$,
$\operatorname{Re}\left(  \alpha\right)  >0$ and$\ \upsilon,\gamma,\mu
,\alpha\in%
\mathbb{C}
$, and gave the relation%
\[
\ _{\upsilon}L_{k}^{\left(  \mu,\gamma\right)  }\left(  z,t\right)  =z^{\mu
}t^{\gamma}E_{\mu+1,\gamma+1,\upsilon}^{(-k)}\left(  z,t\right)
\]
between $_{\upsilon}L_{k}^{\left(  \mu,\gamma\right)  }\left(  z,t\right)  $
and $E_{\mu,\gamma,\upsilon}^{(\alpha)}\left(  z,t\right)  $ in \cite{OzKurt}.

Recently, in \cite{OzEl}, by using the method in \textit{Theorem 3,} the 2D
biorthogonal Hermite Konhauser polynomials have been defined with%
\begin{equation}
\ _{\upsilon}H_{k}^{\left(  \mu\right)  }\left(  z,t\right)  =\sum
\limits_{l=0}^{\left[  k/2\right]  }\sum\limits_{n=0}^{k-l}\frac{\left(
-1\right)  ^{l}\left(  -k\right)  _{2l}\left(  -k\right)  _{l+n}\left(
2z\right)  ^{k-2l}t^{\upsilon n}}{\left(  -k\right)  _{l}\Gamma\left(
\upsilon n+\mu+1\right)  l!n!}, \label{polH}%
\end{equation}
where $\upsilon=1,2,...$ and $\mu>-1$. Further, it has shown that (\ref{polH})
and%
\begin{equation}
Q_{k}\left(  z,t\right)  =H_{k}\left(  z\right)  \sum\limits_{l=0}^{k}%
Y_{l}^{\left(  \mu\right)  }\left(  t;\upsilon\right)  \label{HermiteQ}%
\end{equation}
constitute a pair of bivariate biorthogonal polynomials, where $Y_{l}^{\left(
\mu\right)  }\left(  t;\upsilon\right)  $ are the Konhauser polynomials
defined by (\ref{Ydef}) and $H_{k}\left(  t\right)  $ are the classical
Hermite polynomials (The readers can look the reference in \cite{Cesarano2} for Hermite polynomials). Similar to the papers \cite{OzKurt, OzEl} the
corresponding bivariate Mittag-Leffler functions are defined.

Motivated by papers \cite{OzKurt} and \cite{OzEl}, we derive finite
biorthogonal polynomials in two variables for the first time in this study.
Thus, the notion of finitude is first transferred to biorthogonal polynomials
thanks to this research.

The aim of this paper is to introduce a new family of bivariate biorthogonal
polynomials and the corresponding Mittag-Leffler function in two variables to
the new family. Also, integral and operational representation, fractional
calculus operators, Laplace transform and Fourier transform are obtained for
the new Mittag-Leffler functions and, therefore, for the finite bivariate
biorthogonal polynomials. Further, we give a limit relation for the finite
biorthogonal I - Konhauser polynomials.

On the other hand, by modifying the finite biorthogonal polynomials and the
corresponding Mittag-Leffler functions with two new parameters, we present a
set having the semigroup property.

The general method enabling construct bivariate biorthogonal polynomials using
a univariate biorthogonal and a univariate orthogonal polynomials is reminded
in the first section. In Section 2, via the general method, we present a new
pair of the finite biorthogonal polynomials,$\ _{K}I_{k;\upsilon}^{\left(
p,q\right)  }\left(  z,t\right)  $ and$\ _{K}\mathcal{I}_{k;\upsilon}^{\left(
p,q\right)  }\left(  z,t\right)  $, with the help of the finite orthogonal
polynomials $I_{k}^{\left(  p\right)  }\left(  t\right)  $ and the Konhauser
polynomials, and define the corresponding Mittag-Leffler functions. Thus, we
present a connection between the polynomials$\ _{K}I_{k;\upsilon}^{\left(
p,q\right)  }\left(  z,t\right)  $ and the corresponding Mittag-Leffler
functions $E_{q,\upsilon}^{\left(  \gamma_{1};\gamma_{2};\gamma_{3};\gamma
_{4}\right)  }\left(  z,t\right)  $. Then, we obtain their operational and
integral representation, and compute their Laplace transform. Also, we
consider the fractional calculus approach for these families. After calculating Fourier
transform for the finite biorthogonal polynomials, we give a new set of finite
biorthogonal functions in the light of Parseval identity. In the third
section, to create fractional calculus operators, we modify finite bivariate
biorthogonal polynomials and define the modified Mittag-Leffler functions
corresponding to this modified polynomials. Several properties of the modified
version are studied, and an integral equation and an integral operator are
also constructed.

To realize these objections, let's remind polynomials $I_{k}^{\left(  p\right)
}\left(  t\right)  $.\bigskip

Recall that polynomials $I_{k}^{\left(  p\right)  }\left(  t\right)  $
satisfying the equation%
\[
\left(  1+t^{2}\right)  y_{k}^{\prime\prime}\left(  t\right)  +\left(
3-2p\right)  ty_{k}^{\prime}\left(  t\right)  -k\left(  k+2-2p\right)
y_{k}\left(  t\right)  =0,
\]
are defined by \cite{Masjed}%
\begin{equation}
I_{k}^{\left(  p\right)  }\left(  t\right)  =k!\sum_{l=0}^{\left[  k/2\right]
}\frac{\left(  -1\right)  ^{l}\Gamma\left(  p\right)  \left(  2t\right)
^{k-2l}}{l!\left(  k-2l\right)  !\Gamma\left(  p-k+l\right)  }, \label{Idef}%
\end{equation}
and the following finite orthogonality condition%
\begin{equation}
\int\limits_{-\infty}^{\infty}\left(  1+t^{2}\right)  ^{-\left(  p-1/2\right)
}I_{k}^{\left(  p\right)  }\left(  t\right)  I_{s}^{\left(  p\right)  }\left(
t\right)  dt=\frac{k!2^{2\left(  p-1\right)  }\Gamma^{2}\left(  p\right)
\delta_{k,s}}{\left(  p-k-1\right)  \Gamma\left(  2p-k-1\right)  }
\label{Iort}%
\end{equation}
holds if and only if $p>\max\left\{  k,s\right\}  +1$, \cite{Masjed}. Here, $\delta_{k,s}$ is
the Kronecker's delta.\bigskip

The general method for deriving 2D biorthogonal polynomials is given by
\"{O}zarslan and Elidemir in the theorem below \cite{OzEl}:

\begin{theorem}
Assume that $D_{s}\left(  z\right)  $ and $F_{k}\left(  z\right)  $ are
biorthogonal polynomials with respect to the fundamental polynomials $d\left(
z\right)  $ and $f\left(  z\right)  $, respectively, and the weight function
$\rho_{2}\left(  z\right)  $ over the interval $\left(  \beta_{1},\beta
_{2}\right)  $. That is, they have the biorthogonality relation%
\[
\int\limits_{\beta_{1}}^{\beta_{2}}\rho_{2}\left(  z\right)  D_{s}\left(
z\right)  F_{k}\left(  z\right)  dz=J_{k,s}:=\left\{
\genfrac{}{}{0pt}{}{0,\ k\neq s}{J_{k,k},\ \ \ k=s}%
\right.  .
\]
Also, let%
\[
K_{k}\left(  z\right)  =\sum\limits_{i=0}^{k}E_{k,i}\left(  d\left(  z\right)
\right)  ^{i}%
\]
with%
\[
\int\limits_{\alpha_{1}}^{\alpha_{2}}\rho_{1}\left(  z\right)  K_{s}\left(
z\right)  K_{k}\left(  z\right)  dz=\left\Vert K_{k}\right\Vert ^{2}%
\delta_{k,s}.
\]
Then, for the weight function $\rho_{1}\left(  z\right)  \rho_{2}\left(
t\right)  $, the bivariate polynomials%
\begin{equation}
P_{k}\left(  z,t\right)  =\sum\limits_{r=0}^{k}\frac{E_{k,r}}{J_{k-r,k-r}%
}\left(  d\left(  z\right)  \right)  ^{r}D_{k-r}\left(  t\right)  \label{PolP}%
\end{equation}
and%
\begin{equation}
Q_{k}\left(  z,t\right)  =K_{k}\left(  z\right)  \sum\limits_{j=0}^{k}%
F_{j}\left(  t\right)  \label{PolQ}%
\end{equation}
are biorthogonal on $\left(  \alpha_{1},\alpha_{2}\right)  \times\left(
\beta_{1},\beta_{2}\right)$, \cite{OzEl}.
\end{theorem}

\section{The finite bivariate biorthogonal I - Konhauser polynomials}

In the present section, we first introduce the finite biorthogonal I - Konhauser
polynomials$\ _{K}I_{k;\upsilon}^{\left(  p,q\right) }\left( z,t\right)  $
via Konhauser polynomials and the finite univariate orthogonal polynomials
$I_{k}^{\left(  p\right)  }\left(  t\right)  $. Then, we present the bivariate
I - Konhauser Mittag-Leffler functions $E_{q,\upsilon}^{\left(  \gamma
_{1};\gamma_{2};\gamma_{3};\gamma_{4}\right)  }\left(  z,t\right)  $.

\begin{definition}
The definiton of the finite I - Konhauser polynomials in two variables is
given by the formula%
\begin{equation}
\ _{K}I_{k;\upsilon}^{\left(  p,q\right)  }\left(  z,t\right) =\left(-1\right)^{k}\left(
1-p\right)_{k}\sum_{l=0}^{\left[  k/2\right]
}\sum_{n=0}^{k-l}\frac{\left(  -1\right)  ^{l}\left(  -k\right)  _{2l}\left(
l-k\right)  _{n}\left(  2z\right)  ^{k-2l}t^{\upsilon n}}{\left(  p-k\right)
_{l}\Gamma\left(  \upsilon n+q+1\right)  n!l!}, \label{IKdef}%
\end{equation}
where $\upsilon=1,2,...$ and $q>-1$, $p>\max\left\{  k\right\}  +1$.
\end{definition}

\begin{remark}
Under choices $q=0$ and $t=0$, the polynomials (\ref{IKdef}) reduce the polynomials
$I_{k}^{\left(  p\right)  }\left(  z\right)  $ defined in (\ref{Idef})\ for
$\upsilon=1$. That is%
\[
\ _{K}I_{k;1}^{\left(  p,0\right)  }\left(  z,0\right)  =I_{k}^{\left(
p\right)  }\left(  z\right)  .
\]

\end{remark}

\begin{theorem}
The polynomials (\ref{IKdef}) may be written in terms of polynomials
$Z_{k}^{\left(  \gamma\right)  }\left(  t;\upsilon\right)  $ given by
(\ref{Zdef})\ as follow%
\begin{equation}
\ _{K}I_{k;\upsilon}^{\left(  p,q\right)  }\left(  z,t\right)  =\left(
-1\right)  ^{k}\left(  1-p\right)  _{k}\sum_{l=0}^{\left[  k/2\right]  }%
\frac{\left(  -1\right)  ^{l}\left(  -k\right)  _{2l}\left(  k-l\right)
!\left(  2z\right)  ^{k-2l}Z_{k-l}^{\left(  q\right)  }\left(  t;\upsilon
\right)  }{\left(  p-k\right)  _{l}\Gamma\left(  q+\upsilon\left(  k-l\right)
+1\right)  l!}. \label{IZrel}%
\end{equation}

\end{theorem}

\begin{proof}
The series representations (\ref{IKdef}) and (\ref{Zdef}) completes the proof.
\end{proof}

The second set of the finite biorthogonal I - Konhauser polynomials is of form%
\begin{equation}
\ _{K}\mathcal{I}_{k;\upsilon}^{\left(  p,q\right)  }\left(  z,t\right)
=I_{k}^{\left(  p\right)  }\left(  z\right)  \sum\limits_{j=0}^{k}%
Y_{j}^{\left(  q\right)  }\left(  t;\upsilon\right)  , \label{Qdef}%
\end{equation}
where $I_{k}^{\left(  p\right)  }\left(  z\right)  $ is the finite orthogonal
polynomials defined by (\ref{Idef}) and $Y_{j}^{\left(  q\right)  }\left(
t;\upsilon\right)  $ is the second set of Konhauser polynomials in (\ref{Ydef}).

\begin{corollary}
By \textit{Theorem 3}, polynomials (\ref{IKdef}) and (\ref{Qdef}) constitute
the finite biorthogonal polynomial pair with respect to the weight function
$\rho\left(  z,t\right)  =\frac{e^{-t}t^{q}}{\left(  1+z^{2}\right)  ^{p-1/2}%
}$ over $\left(  -\infty,\infty\right)  \times\left(  0,\infty\right)  $.
\end{corollary}

\begin{theorem}
The polynomials $_{K}I_{k;\upsilon}^{\left(  p,q\right)  }\left(  z,t\right)
$ have the following finite biorthogonality relation%
\begin{align}
&  \int\limits_{-\infty}^{\infty}\int\limits_{0}^{\infty}\ _{K}I_{k;\upsilon
}^{\left(  p,q\right)  }\left(  z,t\right)  \ _{K}\mathcal{I}_{s;\upsilon
}^{\left(  p,q\right)  }\left(  z,t\right)  e^{-t}t^{q}\left(  1+z^{2}\right)
^{-\left(  p-1/2\right)  }dtdz  =\frac{k!2^{2\left(p-1\right)}  \Gamma^{2}\left(
p\right)  }{\left(  p-k-1\right)  \Gamma\left(  2p-k-1\right) }\delta_{k,s},\label{IKort}
\end{align}
under conditions $s,k=0,1,...,S<p-1$ and$\ q>-1$.
\end{theorem}

\begin{proof}
Without loss of generality, let $s \ge k$. Using representations (\ref{IZrel}) and (\ref{Qdef}), we get%
\begin{align*}
&  \int\limits_{-\infty}^{\infty}\int\limits_{0}^{\infty}e^{-t}t^{q}\left(
1+z^{2}\right)  ^{-\left(  p-1/2\right)  }\ _{K}I_{k;\upsilon}^{\left(
p,q\right)  }\left(  z,t\right)  \ _{K}\mathcal{I}_{s;\upsilon}^{\left(
p,q\right)  }\left(  z,t\right)  dzdt\\
&  =\frac{2^{k}\left(  1-p\right)  _{k}}{\left(  -1\right)  ^{k}}\sum
_{l=0}^{\left[  k/2\right]  }\frac{\left(  -1/4\right)  ^{l}\left(  -k\right)
_{2l}}{l!\left(  p-k\right)  _{l}}\int\limits_{-\infty}^{\infty}\left(
1+z^{2}\right)  ^{-\left(  p-1/2\right)  }z^{k-2l}I_{s}^{\left(  p\right)
}\left(  z\right)  dz\\
&  \times\frac{\left(  k-l\right)  !}{\Gamma\left(  \upsilon\left(
k-l\right)  +q+1\right)  }\int\limits_{0}^{\infty}\sum_{j=0}^{s}Y_{j}^{\left(
q\right)  }\left(  t;\upsilon\right)  Z_{k-l}^{\left(  q\right)  }\left(
t;\upsilon\right)  e^{-t}t^{q}dt\\
&  =\int\limits_{-\infty}^{\infty}\left(  1+z^{2}\right)  ^{-\left(
p-1/2\right)  }I_{k}^{\left(  p\right)  }\left(  z\right)  I_{s}^{\left(
p\right)  }\left(  z\right)  dz.
\end{align*}
From (\ref{Iort}), the proof is finished.
\end{proof}

The double hypergeometric functions \cite{SD} is defined as follow%

\begin{equation*}
F_{k,s,r}^{m,p,q}\left(
\genfrac{}{}{0pt}{}{\left\{  a_{j}\right\}  _{j=1}^{m}:\left\{  b_{j}\right\}
_{j=1}^{p};\left\{  c_{j}\right\}  _{j=1}^{q};}{\left\{  d_{j}\right\}
_{j=1}^{k}:\left\{  f_{j}\right\}  _{j=1}^{s};\left\{  g_{j}\right\}
_{j=1}^{r};}%
z;t\right)  =\sum\limits_{n,l=0}^{\infty}\frac{\prod\limits_{j=1}^{m}\left(
a_{j}\right)  _{l+n}\prod\limits_{j=1}^{p}\left(  b_{j}\right)  _{n}%
\prod\limits_{j=1}^{q}\left(  c_{j}\right)  _{l}\ z^{l}t^{n}}{\prod
\limits_{j=1}^{k}\left(  d_{j}\right)  _{l+n}\prod\limits_{j=1}^{s}\left(
f_{j}\right)  _{n}\prod\limits_{j=1}^{r}\left(  g_{j}\right)  _{l}\ l!n!}.
\end{equation*}

\begin{theorem}
The polynomials $_{K}I_{k;\upsilon}^{\left(  p,q\right)  }\left(  z,t\right)
$ may be expressed with the help of the double hypergeometric functions%
\[
\ _{K}I_{k;\upsilon}^{\left(  p,q\right)  }\left(  z,t\right)  =\frac{\left(
1-p\right)  _{k}\left(  -2z\right)  ^{k}}{\Gamma\left(  q+1\right)
}F_{0,\upsilon,2}^{1,0,2}\left(
\genfrac{}{}{0pt}{}{-k:-;\Delta\left(  2;-k\right)  ;}{-:\Delta\left(
\upsilon;q+1\right)  ;-k,p-k;}%
\frac{-1}{z^{2}};\left(  \frac{t}{\upsilon}\right)  ^{\upsilon}\right) ,
\]
where $\Delta\left(  c;\alpha\right)  $ denotes the $c$ parameters
$\frac{\alpha}{c},\frac{\alpha+1}{c},...,\frac{\alpha+c-1}{c}$.
\end{theorem}

\begin{proof}
Using $\left(  1+\gamma\right)  _{\upsilon s}=\upsilon^{\upsilon s}%
\prod\limits_{j=0}^{\upsilon-1}\left(  \frac{\gamma+1+j}{\upsilon}\right)
_{s}$ in definition (\ref{IKdef}), we have the desired.
\end{proof}

\bigskip

Recently, there are lots of investigations on the bivariate variants of
Mittag-Leffler functions \cite{KFO,OF,EOB,Ozarslan,IFO,FKO,OF2}. For example,
in \cite{KFO} two unified families gathering known several and various
Mittag-Leffler functions in two variables were presented. In \cite{EOB},
studying the fractional integral and derivative operators with kernels
including the Mittag-Leffler functions in two variables, specific fractional
Cauchy-type problems comprising these operators were discussed. Also, an
analog of bivariate Mittag-Leffler functions was defined in \cite{FKO}. As
another example, Ozarslan presents a singular integral equation in the kernel
and obtained the solution in terms of the multivariate Mittag-Leffler
functions \cite{Ozarslan}. On the other hand, taking into consideration the
multivariate Mittag-Leffler function in the kernel \cite{OF2}, the authors
investigated the fractional calculus properties with the aid of series approach.

\bigskip

In the present paper, let us define the I - Konhauser Mittag-Leffler functions
$E_{\gamma_{3},\gamma_{4};q;\upsilon}^{\left(  \gamma_{1};\gamma_{2}\right)
}\left(  z,t\right)  $ corresponding to the polynomials $_{K}I_{k;\upsilon
}^{\left(  p,q\right)  }\left(  z,t\right)  $, and give the relation between them.

\begin{definition}
The I - Konhauser Mittag-Leffler functions $E_{\gamma_{3},\gamma_{4};q;\upsilon}^{\left(  \gamma_{1};\gamma_{2}\right)
}\left(  z,t\right)  $ is
defined by%
\begin{equation}
E_{\gamma_{3},\gamma_{4};q;\upsilon}^{\left(  \gamma_{1};\gamma_{2}\right)
}\left(  z,t\right)  =\sum_{l=0}^{\infty}\sum_{n=0}^{\infty}\frac{\left(
-1\right)  ^{n}\left(  \gamma_{1}\right)  _{2n}\left(  \gamma_{2}\right)
_{l+n}z^{n}t^{\upsilon l}}{\left(  \gamma_{3}\right)  _{n}\left(  \gamma
_{4}\right)  _{n}\Gamma\left(  q+\upsilon l\right)  n!l!} \label{Edef}%
\end{equation}
for $q,\upsilon,\gamma_{1},\gamma_{2},\gamma_{3},\gamma_{4}\in%
\mathbb{C},
\operatorname{Re}
\left(  q\right)  >0, \operatorname{Re}\left(  \upsilon\right)  >0$ and $\gamma_{3},\gamma_{4}\notin \mathbb{Z}^-\cup\{0\}.$
\end{definition}

\begin{remark}
The series in (\ref{Edef}) is a special case of the function given in
\cite{SD}. According to the convergence condition stated in \cite{SD}, the
series in (\ref{Edef}) converges for $-\frac{1}{4}<z<\frac{1}{4}$
and$\ -\infty<t<\infty$ thanks to the above conditions.
\end{remark}

By taking $z=0$ or $\gamma_{1}=0$ in (\ref{Edef}), we
write%
\[
E_{\gamma_{3},\gamma_{4};q;\upsilon}^{\left(  \gamma_{1};\gamma_{2}\right)
}\left(  0,t\right)=E_{\gamma_{3},\gamma_{4};q;\upsilon}^{\left(  0;\gamma_{2}\right)
}\left(  z,t\right)
=E_{\upsilon,q}^{\left(  \gamma_{2}\right)  }\left(  t^{\upsilon}\right)  ,
\]
where%
\begin{equation}
E_{p,q}^{\left(  \gamma\right)  }\left(  t\right)  =\sum_{l=0}^{\infty}%
\frac{\left(  \gamma\right)  _{l}\ t^{l}}{l!\ \Gamma\left(  pl+q\right) }\label{E2def}
\end{equation}
is the Mittag-Leffler function introduced by Prabhakar \cite{Prab} for
$\operatorname{Re}\left(  p\right)  >0,\operatorname{Re}\left(  q\right)
>0$ and $p,q,\gamma\in%
\mathbb{C}
$.

On the other hand, if we specially choose $\gamma_{3}=\frac{\gamma_{1}}{2},\gamma_{4}=\frac{\gamma_{1}+1}{2}$ and $t=0$, we get
\[
E_{\frac{\gamma_{1}}{2},\frac{\gamma_{1}+1}{2};q;\upsilon}^{\left(  \gamma_{1};\gamma_{2}\right)
}\left(  z,0\right)  =
\frac{1}{\Gamma\left(  q\right) } \left(  1+4z\right)^{-\gamma_{2}}.
\]

\begin{corollary}
There exist the relation%
\begin{equation}
_{K}I_{k;\upsilon}^{\left(  p,q\right)  }\left(  z,t\right)  =\left(
1-p\right)  _{k}\left(  -2z\right)  ^{k}E_{-k,p-k;q+1;\upsilon}^{\left(
-k;-k\right)  }\left(  \frac{1}{4z^{2}},t\right)  \label{IErel}%
\end{equation}
between the I - Konhauser polynomials and the corresponding Mittag-Leffler function.

\begin{proof}
After choosing $\gamma_{2}=\gamma_{3}$ and then $\gamma_{2}=-k$ in the definition (\ref{Edef}),
we compare the obtained with the definition.
\end{proof}
\end{corollary}
\bigskip
\subsection{Integral and operational representations of the bivariate I -
Konhauser Mittag-Leffler functions and the finite bivariate I - Konhauser
polynomials}

In this subsection the integral and operational representations of the
bivariate I - Konhauser Mittag-Leffler functions $E_{\gamma_{3},\gamma_{4};q;\upsilon}^{\left(  \gamma_{1};\gamma_{2}\right)
}\left(  z,t\right)  $ are
obtained and, using these results, we derive the operational and integral
representations for polynomials $_{K}I_{k;\upsilon}^{\left(  p,q\right)
}\left(  z,t\right)  $.

We know that%
\[
D_{t}^{-1}f\left(  t\right)  =\int\limits_{0}^{t}f\left(  \xi\right)  d\xi.
\]

\begin{theorem}
For the bivariate I - Konhauser Mittag-Leffler functions $E_{\gamma_{3},\gamma_{4};q;\upsilon}^{\left(  \gamma_{1};\gamma_{2}\right)
}\left(
z,t\right)  $ the following operational representation holds:%
\begin{equation}
E_{\gamma_{3},\theta;q;\upsilon}^{\left(  \gamma_{1};\gamma_{2}\right)
}\left(  z,t\right)  =\frac{z^{1-\theta}t^{1-q}}{\left(  1-D_{t}^{-\upsilon
}\right)  ^{\gamma_{2}}}\ _{3}F_{1}\left[
\genfrac{}{}{0pt}{}{\frac{\gamma_{1}}{2},\frac{\gamma_{1}+1}{2},\gamma
_{2}}{\gamma_{3}}%
;\frac{-4}{D_{z}\left(  1-D_{t}^{-\upsilon}\right)  }\right]  \left\{
\frac{z^{\theta-1}t^{q-1}}{\Gamma\left(  q\right)  }\right\}  , \label{Eoprel}%
\end{equation}
where $_{3}F_{1}$ is the special case of the generalized hypergeometric functions.
\end{theorem}

\begin{proof}
Taking into account that $D_{z}^{-n}z^{p-1}=\frac{z^{p+n-1}}{\left(  p\right)
_{n}}$ and $D_{t}^{-\upsilon l}t^{q-1}=\frac{\Gamma\left(  q\right)
t^{q+\upsilon l-1}}{\Gamma\left(  q+\upsilon l\right)  }$, the definition
(\ref{Edef}) is rearranged.
\end{proof}

\begin{corollary}
Polynomials $_{K}I_{k;\upsilon}^{\left(  p,q\right)  }\left(  z,t\right)  $
have the property%
\begin{align*}
\ _{K}I_{k;\upsilon}^{\left(  p,q\right)  }\left(  z,t\right)  =\frac{2^{k}\left(
1-p\right)  _{k}\left(  D_{t}^{-\upsilon}-1\right)  ^{k}}{z^{k-2p+2}t^{q}}%
\ \ \ \ \ \ \ \ \ \ \ \ \ \ \ \ \ \ \ \ \ \ \ \ \ \ \ \ \ \ \ \ \ \ \ \  & \\
\times\ _{2}F_{0}\left[ \frac{-k}{2},\frac{-k+1}{2};-;\frac{2}{z^{3}%
D_{z}\left(  1-D_{t}^{-\upsilon}\right)  }\right]  \left\{  \frac{z^{2\left(
k+1-p\right)  }t^{q}}{\Gamma\left(  q+1\right)  }\right\}  .  &
\end{align*}

\end{corollary}

\begin{proof}
The proof follows by (\ref{IErel}) and (\ref{Eoprel}).
\end{proof}

\bigskip

Now, let's recall the following representations of the Gamma function and the
incomplete Gamma function \cite{Erdelyi, WW}, respectively:%
\begin{equation}
\Gamma\left(  y\right)  =\int\limits_{0}^{\infty}e^{-u}u^{y-1}%
du,\ \ \ \ \ \left(  \operatorname{Re}\left(  y\right)  >0\right)
\label{Gamma}%
\end{equation}
and%
\begin{equation}
\frac{1}{\Gamma\left(  y\right)  }=\frac{1}{2\pi i}\int\limits_{-\infty
}^{0^{+}}e^{u}u^{-y}du,\ \ \ \ \ \left(  \left\vert \arg\left(  u\right)
\right\vert \leq\pi\right)  . \label{PGamma}%
\end{equation}

\begin{theorem}
For parameters $\gamma_{1},\gamma_{3},\gamma_{4}$ with positive real parts, the bivariate I - Konhauser Mittag-Leffler functions have the integral
representation as follow%
\begin{align}
E_{\gamma_{3},\gamma_{4};q;\upsilon}^{\left(  \gamma_{1};\gamma_{2}\right)
}\left(  z,t\right)   &  =-\frac{\Gamma\left(  \gamma_{3}\right)
\Gamma\left(  \gamma_{4}\right)  }{8i\pi^{3}\Gamma\left(  \gamma_{1}\right)
}\int\limits_{0}^{\infty}\int\limits_{-\infty}^{0^{+}}\int\limits_{-\infty
}^{0^{+}}\int\limits_{-\infty}^{0^{+}}e^{-u_{1}+w_{1}+w_{2}+w_{3}}%
u_{1}^{\gamma_{1}-1}w_{1}^{-q}w_{2}^{-\gamma_{3}}w_{3}^{-\gamma_{4}%
}\label{Eintrep1}\\
&  \times\left(  \frac{w_{1}^{\upsilon}-t^{\upsilon}}{w_{1}^{\upsilon}}%
+\frac{u_{1}^{2}z}{w_{2}w_{3}}\right)  ^{-\gamma_{2}}dw_{3}dw_{2}dw_{1}%
du_{1}.\nonumber
\end{align}

\end{theorem}

\begin{proof}
Using (\ref{Gamma}) and (\ref{PGamma}), we obtain (\ref{Eintrep1}).
\end{proof}

\begin{corollary}
Polynomials $_{K}I_{k;\upsilon}^{\left(  p,q\right)  }\left(  z,t\right)  $
have the integral representation as follow%
\begin{align*}
_{K}I_{k;\upsilon}^{\left(  p,q\right)  }\left(  z,t\right)   &
=\frac{ \sqrt{\pi}  z^{k} \ \Gamma\left(  p\right) }{16\pi^{4}}\int\limits_{-\infty}^{0^{+}} \int\limits_{-\infty}^{0^{+}}%
\int\limits_{-\infty}^{0^{+}}\int\limits_{-\infty}^{0^{+}}\int\limits_{0}^{\infty}e^{-u_{1}+w_{1}+w_{2}%
+w_{3}+w_{4}}u_{1}^{k}w_{1}^{-q-1}
\label{Iintrep}\\
&  \times w_{2}^{-\frac{k}{2}-1}w_{3}^{-\frac{k+1}{2}}w_{4}^{k-p}\left( \frac{w_{1}^{\upsilon
}-t^{\upsilon}}{w_{1}^{\upsilon}}- \frac{w_{2}w_{3}}{w_{4}u_{1}z^{2}}\right)  ^{k}du_{1}dw_{1}dw_{2}dw_{3}dw_{4}%
.\nonumber
\end{align*}

\end{corollary}

\begin{proof}
For the proof, (\ref{Gamma}) and (\ref{PGamma}) are used.
\end{proof}
\bigskip
\subsection{Fractional calculus operators and Laplace transform for the
bivariate I - Konhauser Mittag-Leffler functions and the finite bivariate I -
Konhauser polynomials}

In this subsection, Laplace transform,
fractional integral and fractional
derivative operators of $E_{\gamma_{3},\gamma_{4};q;\upsilon}^{\left(  \gamma_{1};\gamma_{2}%
\right)  }\left(  z,t\right)  $ and thus, of
$_{K}I_{k;\upsilon}^{\left(  p,q\right)  }\left(  z,t\right)  $ are computed.

The Laplace transform for a function $h$ in one variable is defined as follow \cite{KG}%
\begin{equation}
L[h]\left(  a\right)  =\int\limits_{0}^{\infty}e^{-a\xi}h\left(  \xi\right)
d\xi,\ \operatorname{Re}\left( a\right)  >0. \label{LaplaceDef}%
\end{equation}

\begin{theorem}
The Laplace transform of $E_{\gamma_{3},\gamma_{4};q;\upsilon}^{\left(  \gamma_{1};\gamma_{2}%
\right)  }\left(  z,t\right)  $ is given by%
\begin{equation}
L\left\{  t^{q-1}E_{\gamma_{3},\gamma_{4};q;\upsilon}^{\left(  \gamma_{1};\gamma_{2}%
\right)  }\left(  z,wt\right)  \right\}  =\frac{1}{a^{q}%
}\left(  \frac{a^{\upsilon}-w^{\upsilon}}{a^{\upsilon}}\right)  ^{-\gamma_{2}%
}\ _{3}F_{2}\left[
\genfrac{}{}{0pt}{}{\frac{\gamma_{1}}{2},\frac{\gamma_{1}+1}{2},\gamma
_{2}}{\gamma_{3},\gamma_{4}}%
;\frac{-4za^{\upsilon}}{a^{\upsilon}-w^{\upsilon}}\right]  \label{Elap}%
\end{equation}
for $\left\vert \frac{w^{\upsilon}}{a^{\upsilon}}\right\vert <1$.
\end{theorem}

\begin{proof}
Direct calculation yield%
\begin{align*}
&  L\left\{  t^{q-1}E_{\gamma_{3},\gamma_{4};q;\upsilon}^{\left(  \gamma_{1};\gamma_{2}%
\right)  }\left(  z,wt\right)  \right\}  =\int\limits_{0}%
^{\infty}e^{-at}t^{q-1}\sum_{l=0}^{\infty}\sum_{n=0}^{\infty}\frac{\left(
-1\right)  ^{l}\left(  \gamma_{1}\right)  _{2l}\left(  \gamma_{2}\right)
_{n+l}z^{l}\left(  wt\right)  ^{\upsilon n}}{\left(  \gamma_{3}\right)
_{l}\left(  \gamma_{4}\right)  _{l}\Gamma\left(  q+\upsilon n\right)
n!l!}dt\\
&  =\frac{1}{a^{q}}\sum_{l=0}^{\infty}\frac{\left(  \frac{\gamma_{1}}%
{2}\right)  _{l}\left(  \frac{\gamma_{1}+1}{2}\right)  _{l}\left(  \gamma
_{2}\right)  _{l}\left(  -4z\right)  ^{l}}{\left(  \gamma_{3}\right)
_{l}\left(  \gamma_{4}\right)  _{l}l!}\sum_{n=0}^{\infty}\frac{\left(
\gamma_{2}+l\right)  _{n}}{n!}\left(  \frac{w^{\upsilon}}{a^{\upsilon}%
}\right)  ^{n}\\
&  =\frac{1}{a^{q}}\left(  \frac{a^{\upsilon}-w^{\upsilon}}{a^{\upsilon}%
}\right)  ^{-\gamma_{2}}\sum_{l=0}^{\infty}\frac{\left(  \frac{\gamma_{1}}%
{2}\right)  _{l}\left(  \frac{\gamma_{1}+1}{2}\right)  _{l}\left(  \gamma
_{2}\right)  _{l}}{\left(  \gamma_{3}\right)  _{l}\left(  \gamma_{4}\right)
_{l}l!}\left(  \frac{-4za^{\upsilon}}{a^{\upsilon}-w^{\upsilon}}\right)
^{l}\\
&  =\frac{1}{a^{q}}\left(  \frac{a^{\upsilon}-w^{\upsilon}}{a^{\upsilon}%
}\right)  ^{-\gamma_{2}}\ _{3}F_{2}\left[
\genfrac{}{}{0pt}{}{\frac{\gamma_{1}}{2},\frac{\gamma_{1}+1}{2},\gamma
_{2}}{\gamma_{3},\gamma_{4}}%
;\frac{-4za^{\upsilon}}{a^{\upsilon}-w^{\upsilon}}\right]  ,
\end{align*}
where operations of integral and sum can be replaced because of the uniform
convergence of $E_{\gamma
_{3},\gamma_{4};q;\upsilon}^{\left(  \gamma_{1};\gamma_{2}\right)  }\left(  z,t\right)  $.
\end{proof}

\begin{corollary}
By choosing $\gamma_{3}=\frac{\gamma_{1}}{2}$ and $\gamma_{4}=\frac{\gamma
_{1}+1}{2}$, we have the following Laplace transform of $E_{\gamma_{3},\gamma_{4};q;\upsilon
}^{\left(  \gamma_{1};\gamma_{2}\right)  }\left(
z,t\right)  $ for $\left\vert \frac{w^{\upsilon}}{a^{\upsilon}}\right\vert
<1$:%
\[
L\left\{  t^{q-1}E_{\frac
{\gamma_{1}}{2},\frac{\gamma_{1}+1}{2};q;\upsilon}^{\left(  \gamma_{1};\gamma_{2}\right)  }\left(  z,wt\right)  \right\}
=\frac{1}{a^{q}}\left(  4z+\frac{a^{\upsilon}-w^{\upsilon}}{a^{\upsilon}%
}\right)  ^{-\gamma_{2}}.
\]

\end{corollary}

\begin{corollary}
For $\left\vert \frac{w^{\upsilon}}{a^{\upsilon}}\right\vert <1$, the Laplace
transform of polynomials $_{K}I_{k;\upsilon}^{\left(  p,q\right)  }\left(
z,t\right)  $ is%
\[
L\left\{  t^{q}\ _{K}I_{k;\upsilon}^{\left(  p,q\right)  }\left(  z,wt\right)
\right\}  =\frac{\left(  1-p\right)  _{k}\left(  -2z\right)  ^{k}}{a^{q+1}%
}\left(  \frac{a^{\upsilon}-w^{\upsilon}}{a^{\upsilon}}\right)  ^{k}%
\ _{3}F_{2}\left[
\genfrac{}{}{0pt}{}{-k,\frac{-k}{2},\frac{-k+1}{2}}{-k,p-k}%
;\frac{-a^{\upsilon}}{z^{2}\left(  a^{\upsilon}-w^{\upsilon}\right)  }\right]
.
\]

\end{corollary}

\begin{proof}
The proof can be obtained easily from (\ref{IErel}) and (\ref{Elap}).
\end{proof}

\bigskip

Definition of the Riemann-Liouville fractional integral and derivative are
given by \cite{Kilbas}%
\[
_{z}\mathbb{I}_{a^{+}}^{\sigma}\left(  h\right)  =\frac{1}{\Gamma\left(
\sigma\right)  }\int\limits_{a}^{z}\left(  z-\zeta\right)  ^{\sigma-1}h\left(
\zeta\right)  d\zeta,\ \ \ \ h\in L^{1}\left[  a,b\right]
\]
and%
\[
_{z}D_{a^{+}}^{\sigma}\left(  h\right)  =\left(  \frac{d}{dz}\right)
^{k}\ _{z}\mathbb{I}_{a^{+}}^{k-\sigma}\left(  h\right)  ,\ \ \ \ h\in
C^{k}\left[  a,b\right]  ,
\]
where $[\operatorname{Re}\left(  \sigma\right)  ]$ denotes the integral part
of $\operatorname{Re}\left(  \sigma\right)  $ and $k=\left[  \operatorname{Re}%
\left(  \sigma\right)  \right]  +1$, $z>a$, $\operatorname{Re}\left(
\sigma\right)  >0$, $\sigma\in%
\mathbb{C}
$.\\
To derive the integral, we will need the definition%
\begin{align}
&  \int\limits_{0}^{1}u^{t-1}\left(  1-u\right)  ^{z-1}du  =B\left(  t,z\right)  =B(z,t)=\frac{\Gamma\left(  t\right)  \Gamma\left(
z\right)  }{\Gamma\left(  t+z\right)  },\label{Betadef}
\end{align}
called the Beta integral for $\operatorname{Re}\left(  z\right)  >0$ and
$\operatorname{Re}\left(  t\right)  >0$.

\begin{theorem}
For $\operatorname{Re}\left(  \tau\right)  >0$, the bivariate I - Konhauser
Mittag-Leffler functions have the Riemann-Liouville fractional integral
operator%
\begin{equation}
_{t}\mathbb{I}_{b^{+}}^{\tau}\left[  \left(  t-b\right)  ^{q-1}E_{\gamma_{3},\gamma_{4};q;\upsilon
}^{\left(  \gamma_{1};\gamma_{2}\right)  }\left(
z,w\left(  t-b\right)  \right)  \right]  =\left(  t-b\right)  ^{q+\tau
-1}E_{\gamma_{3},\gamma_{4};q+\tau;\upsilon}^{\left(  \gamma_{1};\gamma_{2} \right)  }\left(  z,w\left(  t-b\right)  \right)  . \label{Efracint}%
\end{equation}

\end{theorem}

\begin{proof}%
\begin{align*}
&  _{t}\mathbb{I}_{b^{+}}^{\tau}\left[  \left(  t-b\right)  ^{q-1}%
E_{\gamma_{3},\gamma_{4};q;\upsilon}^{\left(  \gamma_{1};\gamma_{2}\right)
}\left(  z,w\left(  t-b\right)  \right)  \right] \\
&  =\frac{1}{\Gamma\left(  \tau\right)  }\sum_{l=0}^{\infty}\sum_{n=0}%
^{\infty}\frac{\left(  -1\right)  ^{l}\left(  \gamma_{1}\right)  _{2l}\left(
\gamma_{2}\right)  _{l+n}\ z^{l}w^{\upsilon n}}{\left(  \gamma_{3}\right)
_{l}\left(  \gamma_{4}\right)  _{l}\Gamma\left(  q+\upsilon n\right)
n!l!}\int\limits_{b}^{t}\left(  t-y\right)  ^{\tau-1}\left(  y-b\right)
^{q+\upsilon n-1}dy\\
&  =\frac{\left(  t-b\right)  ^{q+\tau-1}}{\Gamma\left(  \tau\right)  }%
\sum_{l=0}^{\infty}\frac{\left(  \frac{\gamma_{1}}{2}\right)  _{l}\left(
\frac{\gamma_{1}+1}{2}\right)  _{l}\left(  \gamma_{2}\right)  _{l}\left(
-4z\right)  ^{l}}{\left(  \gamma_{3}\right)  _{l}\left(  \gamma_{4}\right)
_{l}l!}\\
&  \times\sum_{n=0}^{\infty}\frac{\left(  \gamma_{2}+l\right)  _{n}}%
{\Gamma\left(  q+\upsilon n\right)  n!}\left(  w\left(  t-b\right)  \right)
^{\upsilon n}\int\limits_{0}^{1}\left(  1-u\right)  ^{\tau-1}u^{q+\upsilon
n-1}du\\
&  =\left(  t-b\right)  ^{q+\tau-1}\sum_{l=0}^{\infty}\frac{\left(
\frac{\gamma_{1}}{2}\right)  _{l}\left(  \frac{\gamma_{1}+1}{2}\right)
_{l}\left(  \gamma_{2}\right)  _{l}\left(  -4z\right)  ^{l}}{\left(
\gamma_{3}\right)  _{l}\left(  \gamma_{4}\right)  _{l}l!}\sum_{n=0}^{\infty
}\frac{\left(  \gamma_{2}+l\right)  _{n}}{\Gamma\left(  q+\tau+\upsilon
n\right)  n!}\left(  w\left(  t-b\right)  \right)  ^{\upsilon n}.
\end{align*}
From definition (\ref{Edef}), we have (\ref{Efracint}) for $\operatorname{Re}%
\left(  \tau\right)  >0$.
\end{proof}

\begin{corollary}
For polynomials (\ref{IKdef}), we have%
\[
_{t}\mathbb{I}_{b^{+}}^{\tau}\left[  \left(  t-b\right)  ^{q}\ _{K}%
I_{k;\upsilon}^{\left(  p,q\right)  }\left(  z,w\left(  t-b\right)  \right)
\right]  =\left(  t-b\right)  ^{q+\tau}\ _{K}I_{k;\upsilon}^{\left(
p,q+\tau\right)  }\left(  z,w\left(  t-b\right)  \right)  ,
\]
where $\operatorname{Re}\left(  \tau\right)  >0$.
\end{corollary}

\begin{proof}
The proof is followed by the definition (\ref{IErel}) and (\ref{Efracint}).
\end{proof}

\begin{theorem}
For $\operatorname{Re}\left(  \tau\right)  \geq0$ and $\operatorname{Re}\left(  q-\tau \right)  \geq0$, the bivariate I - Konhauser
Mittag-Leffler functions have the Riemann-Liouville fractional derivative
operator%
\begin{equation}
\ _{t}D_{b^{+}}^{\tau}\left[  \left(  t-b\right)  ^{q-1}E_{\gamma_{3},\gamma_{4};q;\upsilon
}^{\left(  \gamma_{1};\gamma_{2}\right)  }\left(
z,w\left(  t-b\right)  \right)  \right]  =\left(  t-b\right)  ^{q-\tau
-1}E_{\gamma_{3},\gamma
_{4};q-\tau;\upsilon}^{\left(  \gamma_{1};\gamma_{2}\right)  }\left(  z,w\left(  t-b\right)  \right)  . \label{Efracder}%
\end{equation}

\end{theorem}

\begin{proof}
For $\operatorname{Re}\left(  \tau\right)  \geq0$ and $\operatorname{Re}\left(  q-\tau \right)  \geq0$,%
\begin{align*}
&  \ _{t}D_{b^{+}}^{\tau}\left[  \left(  t-b\right)  ^{q-1}E_{\gamma_{3},\gamma_{4};q;\upsilon
}^{\left(  \gamma_{1};\gamma_{2}\right)  }\left(
z,w\left(  t-b\right)  \right)  \right] \\
&  =\frac{1}{\Gamma\left(  k-\tau\right)  }\sum_{l=0}^{\infty}\sum
_{n=0}^{\infty}\frac{\left(  -1\right)  ^{l}\left(  \gamma_{1}\right)
_{2l}\left(  \gamma_{2}\right)  _{l+n}\ z^{l}w^{\upsilon n}}{\left(
\gamma_{3}\right)  _{l}\left(  \gamma_{4}\right)  _{l}\Gamma\left(  q+\upsilon
n\right)  l!n!}D_{t}^{k}\int\limits_{b}^{t}\left(  t-y\right)  ^{k-\tau
-1}\left(  y-b\right)  ^{q+\upsilon n-1}dy\\
&  =\sum_{l=0}^{\infty}\sum_{n=0}^{\infty}\frac{\left(  -1\right)  ^{l}\left(
\gamma_{1}\right)  _{2l}\left(  \gamma_{2}\right)  _{l+n}\ z^{l}w^{\upsilon
n}}{l!n!\left(  \gamma_{3}\right)  _{l}\left(  \gamma_{4}\right)  _{l}}%
\frac{1}{\Gamma\left(  k-\tau+q+\upsilon n\right)  }D_{t}^{k}\left(
t-b\right)  ^{k-\tau+q+\upsilon n-1}\\
&  =\left(  t-b\right)  ^{q-\tau-1}\sum_{l=0}^{\infty}\sum_{n=0}^{\infty}%
\frac{\left(  -1\right)  ^{l}\left(  \gamma_{1}\right)  _{2l}\left(
\gamma_{2}\right)  _{l+n}\ z^{l}\left(  w\left(  t-b\right)  \right)
^{\upsilon n}}{\left(  \gamma_{3}\right)  _{l}\left(  \gamma_{4}\right)
_{l}\Gamma\left(  q-\tau+\upsilon n\right)  l!n!}.
\end{align*}
Considering (\ref{Edef}), we get the desired.
\end{proof}

\begin{corollary}
For $\operatorname{Re}\left(  \tau\right)  \geq0$ and $\operatorname{Re}\left(  q-\tau \right)  \geq0$, the polynomials$\ _{K}%
I_{k;\upsilon}^{\left(  p,q\right)  }\left(  z,t\right)  $ have the
Riemann-Liouville fractional derivative operator as follow:%
\[
\ _{t}D_{b^{+}}^{\tau}\left[  \left(  t-b\right)  ^{q}\ _{K}I_{k;\upsilon
}^{\left(  p,q\right)  }\left(  z,w\left(  t-b\right)  \right)  \right]
=\left(  t-b\right)  ^{q-\tau}\ _{K}I_{k;\upsilon}^{\left(  p,q-\tau\right)
}\left(  z,w\left(  t-b\right)  \right)  .
\]

\end{corollary}

\begin{proof}
(\ref{IErel}) and (\ref{Efracder}) follows the proof.
\end{proof}
\bigskip
\subsection{Fourier transform of I - Konhauser polynomials}

The Fourier transform and the corresponding Parseval identity for a
two-variable function $d(z,t)$ are defined as \cite{Davies}%
\begin{equation}%
\mathcal{F}%
\left(  d\left(  z,t\right)  \right)  =\int\limits_{-\infty}^{\infty}%
\int\limits_{-\infty}^{\infty}d\left(  z,t\right)  e^{-\left(  \xi_{1}%
z+\xi_{2}t\right)  i}dzdt \label{Fourier}%
\end{equation}
and%
\begin{equation}
\int\limits_{-\infty}^{\infty}\int\limits_{-\infty}^{\infty}\overline{f\left(
z,t\right)  }d\left(  z,t\right)  dzdt=\frac{1}{\left(  2\pi\right)  ^{2}}%
\int\limits_{-\infty}^{\infty}\int\limits_{-\infty}^{\infty}\overline{%
\mathcal{F}%
\left(  f\left(  z,t\right)  \right)  }%
\mathcal{F}%
\left(  d\left(  z,t\right)  \right)  d\xi_{1}d\xi_{2}, \label{Parseval}%
\end{equation}
respectively.

\bigskip

Let's consider the functions%
\[%
\genfrac{\{}{.}{0pt}{}{d\left(  z,t\right)  =e^{at-\frac{\exp t}{2}}\left(
1+e^{2z}\right)  ^{-\left(  b-1/4\right)  }e^{\frac{\left(-1\right)^k}{2}z}\ _{K}I_{k;\upsilon}^{\left(
p,q\right)  }\left(  e^{z},e^{t}\right)  }{f\left(  z,t\right)  =e^{dt-\frac{\exp t}%
{2}}\left(  1+e^{2z}\right)  ^{-\left(  c-1/4\right)  }e^{z-\frac{\left(-1\right)^s}{2}z}\ _{K}%
\mathcal{I}_{s;\upsilon}^{\left(  \alpha,\beta\right)  }\left(   e^{z},e^{t}%
\right)  }%
\]
and calculate the following corresponding Fourier transform, using the
transforms $\exp z=x, \exp t=y$ and the definition (\ref{IKdef}),%
\begin{align*}
&
\mathcal{F}%
\left(  d\left(  z,t\right)  \right)  =\int\limits_{-\infty}^{\infty}%
\int\limits_{-\infty}^{\infty}d\left(  z,t\right)  e^{-i\left(  \xi_{1}%
z+\xi_{2}t\right)}dzdt\\
&  =\int\limits_{-\infty}^{\infty}\int\limits_{-\infty}^{\infty}e^{-i\left(
\xi_{1}z+\xi_{2}t\right)  }e^{at-\frac{\exp t}{2}}\left(
1+e^{2z}\right)  ^{-\left(  b-1/4\right)  }e^{\frac{\left(-1\right)^k}{2}z}\ _{K}I_{k;\upsilon}^{\left(  p,q\right)  }\left(
e^{z},e^{t}\right)  dzdt\\
&  =\left(  -1\right)  ^{k}\left(  1-p\right)  _{k}\sum
_{l=0}^{\left[  k/2\right]  }\sum_{n=0}^{k-l}\frac{\left(  -1\right)
^{l}2^{k-2l}\left(  -k\right)  _{2l}\left(  -\left(  k-l\right)  \right)
_{n}}{l!n!\Gamma\left(  q+\upsilon n+1\right)  \left(  p-k\right)  _{l}}\\
&  \times\int\limits_{0}^{\infty}x^{k-2l-i\xi_{1}-1+\frac{\left(-1\right)^k}{2}}\left(
1+x^{2}\right)  ^{-\left(  b-1/4\right)  }dx\int\limits_{0}^{\infty}%
e^{-\frac{y}{2}}y^{a+\upsilon n-i\xi_{2}-1}dy.
\end{align*}
With the help of the definitions of the Beta and Gamma function, and the
transform $y=2u$, we obtain%
\begin{align}
\mathcal{F}%
\left(  d\left(  z,t\right)  \right)  =\left(
-1\right)  ^{k} \left(  1-p\right)_{k} 2^{k+a-i\xi_{2}} \Gamma\left(  a-i\xi_{2}\right)  \sum_{l=0}^{\left[  k/2\right]  }\sum_{n=0}^{k-l}%
\frac{ \left(-1\right)^l \left(  -k\right)_{2l}\left(  -\left(  k-l\right)  \right)  _{n}}{\left(  p-k\right)  _{l} \ l!n! }\label{Fourier1} \\
\times \frac{\left(  a-i\xi_{2}\right)_{\upsilon n} \left( \frac{1}{4}\right)^l 2^{\upsilon n}}{\Gamma\left(  q+\upsilon
n+1\right) } \int\limits_{0}^{\infty}x^{k-2l-i\xi_{1}-1+\frac{\left(-1\right)^k}{2}}\left(
1+x^{2}\right)  ^{-\left(  b-1/4\right)  }dx \nonumber
\end{align}
for $\operatorname{Re}\left(  a\right) >0$.
Here, we investigate the integral
\begin{align}
I_{k,l} \left(b,\xi_{1}\right)=\int\limits_{0}^{\infty}x^{k-2l-i\xi_{1}-1+\frac{\left(-1\right)^k}{2}}\left(
1+x^{2}\right)  ^{-\left(  b-1/4\right)  }dx. \label{int1}
\end{align}
If $k$ is even, then
\begin{align*}
I_{2m,l} \left(b,\xi_{1}\right)=\int\limits_{0}^{\infty}x^{2\left(m-l-\frac{i\xi_{1}}{2}-\frac{1}{2}+\frac{1}{4}\right)}\left(
1+x^{2}\right)  ^{-\left(  b-1/4\right)  }dx\\
=\frac{1}{2} B\left(m-l-\frac{i\xi_{1}}{2}+\frac{1}{4},b-m+l+\frac{i\xi_{1}}{2}-\frac{1}{2}\right),
\end{align*}
and for odd $k$, we have
\begin{align*}
I_{2m+1,l} \left(b,\xi_{1}\right)=\int\limits_{0}^{\infty}x^{2\left(m-l-\frac{i\xi_{1}}{2}-\frac{1}{4}\right)}\left(
1+x^{2}\right)  ^{-\left(  b-1/4\right)  }dx\\
=\frac{1}{2} B\left(m-l-\frac{i\xi_{1}}{2}+\frac{1}{4},b-m+l+\frac{i\xi_{1}}{2}-\frac{1}{2}\right).
\end{align*}
Thus, the integral (\ref{int1}) is
\begin{align*}
I_{k,l} \left(b,\xi_{1}\right)=\frac{1}{2} B\left(\left[ \frac{k}{2} \right]-l-\frac{i\xi_{1}}{2}+\frac{1}{4},b-\left[ \frac{k}{2} \right]+l+\frac{i\xi_{1}}{2}-\frac{1}{2}\right). 
\end{align*}
So the Fourier transform (\ref{Fourier1}) becomes
\begin{align}%
\mathcal{F}%
\left(  d\left(  z,t\right)  \right)  & =\left(  -1\right)  ^{k+\left[ \frac{k}{2} \right]}%
2^{k-1}\left(  1-p\right)  _{k} \ G_{1}\left(  a,b;\xi_{1},\xi_{2}\right)  \Psi_{1}\left(
k,a,b,p,q,\upsilon;\xi_{1},\xi_{2}\right)  ,\label{F1}
\end{align}
where%
\[
G_{1}\left( a,b;\xi_{1},\xi_{2}\right) = 2^{a-i\xi_{2}} \Gamma\left(
a-i\xi_{2}\right)
B\left( \frac{1}{4}-\frac{i\xi_{1}}{2},b-\frac{1}{2}+\frac{i\xi_{1}}{2}\right) 
\]
and
\begin{align*}
\Psi_{1}\left(  k,a,b,p,q,\upsilon;\xi_{1},\xi_{2}\right)=\frac{\left(\frac{1}{4}-\frac{i\xi_{1}}{2}\right)_{\left[ \frac{k}{2} \right]}}{\left(\frac{3}{2}-b+\frac{i\xi_{1}}{2}\right)_{\left[ \frac{k}{2} \right]}}\sum
_{l=0}^{\left[  k/2\right]  }\sum_{n=0}^{k-l}\frac{\left(  \frac{-k}{2}\right)
_{l} \left(  \frac{-k+1}{2}\right)
_{l} \left(  -\left(  k-l\right)  \right)  _{n} }{\left(  p-k\right)  _{l} \Gamma\left(  q+\upsilon n+1\right)} \\
\times \frac{\left(  a-i\xi_{2}\right)  _{\upsilon n} \left( b-\left[ \frac{k}{2} \right]-\frac{1}{2} +\frac{i\xi_{1}}{2}\right)  _{l} \  2^{\upsilon n}}{\left( \frac{3}{4}- \left[ \frac{k}{2} \right]+\frac{i\xi_{1}}{2}\right)_{l} l!n! },
\end{align*}
where $\operatorname{Re}\left(b\right)>\left[ \frac{k}{2} \right]+\frac{1}{2}$. \\
Similarly,
\begin{align}
\mathcal{F}%
\left( f\left( z,t\right) \right) =\left(
-1\right)  ^{s} \left(  1-\alpha\right)_{s} 2^{s+d-i\xi_{2}} \Gamma\left(  d-i\xi_{2}\right)  \sum_{l=0}^{\left[  s/2\right] }%
\frac{ \left(-1\right)^l \left(  \frac{-s}{2}\right)_{l}\left(  \frac{-s+1}{2}\right)_{l}}{\left(  \alpha-s\right)  _{l} \ l! }\label{Fourier2} \\
\times \int\limits_{0}^{\infty}x^{s-2l-i\xi_{1}-\frac{\left(-1\right)^s}{2}}\left(
1+x^{2}\right)  ^{-\left(  c-1/4\right)  }dx \sum_{j=0}^{s}\sum_{r=0}^{j}\sum_{n=0}^{r} \frac{\left(  -r\right)_{n} \left(  d-i\xi_{2}\right)_{r} \left( \frac{1+\beta+n}{\upsilon}\right)_{j} 2^{r}}{j!r!n!} \nonumber
\end{align}
is calculated for $\operatorname{Re}\left(  d\right) >0$.\\
If we say
\begin{align}
I_{s,l} \left(c,\xi_{1}\right)=\int\limits_{0}^{\infty}x^{s-2l-i\xi_{1}-\frac{\left(-1\right)^s}{2}}\left(
1+x^{2}\right)  ^{-\left(  c-1/4\right)  }dx \label{Isl},
\end{align}
for $s=0,2,4,...$, the integral (\ref{Isl}) gives
\begin{align}
I_{2m,l} \left(c,\xi_{1}\right)=\int\limits_{0}^{\infty}x^{2\left(m-l-\frac{i\xi_{1}}{2}-\frac{1}{4}\right)}\left(
1+x^{2}\right)  ^{-\left(  c-1/4\right)  }dx\\
=\frac{1}{2} B\left(m-l-\frac{i\xi_{1}}{2}+\frac{1}{4},c-m+l+\frac{i\xi_{1}}{2}-\frac{1}{2}\right) \nonumber,
\end{align}
and for $s=1,3,5,...$, it becomes
\begin{align}
I_{2m+1,l} \left(b,\xi_{1}\right)=\int\limits_{0}^{\infty}x^{2\left(m-l-\frac{i\xi_{1}}{2}+\frac{1}{2}+\frac{1}{4}\right)}\left(
1+x^{2}\right)  ^{-\left(  c-1/4\right)  }dx\\
=\frac{1}{2} B\left(m-l-\frac{i\xi_{1}}{2}+\frac{5}{4},c-m+l+\frac{i\xi_{1}}{2}-\frac{3}{2}\right) \nonumber.
\end{align}
Thus, the integral (\ref{Isl}) is
\begin{align}
I_{s,l} \left(c,\xi_{1}\right)=\frac{1}{2} B\left(\left[ \frac{s+1}{2} \right]-l-\frac{i\xi_{1}}{2}+\frac{1}{4},c-\left[ \frac{s+1}{2} \right]+l+\frac{i\xi_{1}}{2}-\frac{1}{2}\right). 
\end{align} \\
On the other hand, for $\operatorname{Re}\left(  d\right)  >0$ and $ \operatorname{Re} \left( c \right)> \left[ \frac{s+1}{2} \right] +\frac{1}{2}$, Fourier transform of $f\left( z,t\right)$ may be calculated of form
\begin{align}%
\mathcal{F}%
\left(  f\left(  z,t\right)  \right)   &  =\left(-1\right)^{s+\left[ \frac{s+1}{2} \right]} 2^{s-1} \left(1-\alpha\right)_{s} \ G_{2}\left(  d,c;\xi_{1},\xi_{2}\right)  \Psi_{2}\left(
s,d,c,\alpha,\beta,\upsilon;\xi_{1},\xi_{2}\right)  , \label{F2}
\end{align}
where%
\[
G_{2}\left(  d,c;\xi_{1},\xi_{2}\right)  = 2^{d-i\xi_{2}} \Gamma\left(
d-i\xi_{2}\right)
 B\left( \frac{1}{4}-\frac{i\xi_{1}}{2}, c- \frac{1}{2}+\frac{i\xi_{1}}{2}\right) 
\]
and%
\begin{align*}
\Psi_{2}\left(s,d,c,\alpha,\beta,\upsilon;\xi_{1},\xi_{2}\right)
=\frac{\left(  \frac{1}{4}-\frac{i\xi_{1}}{2}\right)_{\left[ \frac{s+1}{2} \right]}}{\left(  \frac{3}{2}-c-\frac{i\xi_{1}}{2}\right)_{\left[ \frac{s+1}{2} \right]}} \sum_{l=0}^{\left[  s/2\right]  }\frac{\left(  \frac{-s}{2}\right)
_{l} \left(  \frac{-s+1}{2}\right)
_{l} \left( c-\left[ \frac{s+1}{2} \right]-\frac{1}{2} +\frac{i\xi_{1}}{2}\right) _{l} }{\left( \frac{3}{4}-\left[ \frac{s+1}{2} \right]+\frac{i\xi_{1}}{2}\right)_{l} \left(  \alpha-s\right) _{l} l!}\\
\times \sum_{j=0}^{s}\sum_{r=0}^{j}\sum_{n=0}^{r}\frac{\left(-r\right)_{n}\left(  d-i\xi_{2}\right)  _{r} 2^{r}}{j!r!n!}\left(  \frac{\beta
+n+1}{\upsilon}\right)  _{j}.
\end{align*}
\bigskip
\subsection{A class of finite biorthogonal functions derived from Fourier
transform for the finite I\ - Konhauser polynomials}

\textbf{Case 1:} When $k$ and $s$ are both odd or even, replacing (\ref{F1}) and (\ref{F2}) in (\ref{Parseval}) and taking $e^{z}=x$ and $e^{t}=y$,
we obtain%
\begin{align}
&  \int\limits_{0}^{\infty}\int\limits_{0}^{\infty}e^{-y}y^{a+d-1}%
\left(  1+x^{2}\right)  ^{-\left(  b+c-1/2\right)  }\ _{K}I_{k;\upsilon
}^{\left(  p,q\right)  }\left(  x,y\right)  \ _{K}\mathcal{I}_{s;\upsilon
}^{\left(  \alpha,\beta\right)  }\left(  x,y\right)  dxdy\label{eq1}\\
& =\frac{\left(  -1\right)  ^{k+s+\left[ \frac{k}{2} \right]+\left[ \frac{s+1}{2} \right]} 2^{k+s-4} \left(  1-p\right)  _{k} \left(  1-\alpha\right)  _{s}}{\pi^2}
\int\limits_{-\infty}^{\infty}\int\limits_{-\infty}^{\infty}G_{1}\left(
a,b;\xi_{1},\xi_{2}\right) \overline{G_{2}\left(  d,c;\xi_{1},\xi_{2}\right)  } \nonumber\\
&  \times\Psi_{1}\left(
k,a,b,p,q,\upsilon;\xi_{1},\xi_{2}\right)  \overline{\Psi_{2}\left(
s,d,c,\alpha,\beta,\upsilon;\xi_{1},\xi_{2}\right)  }d\xi_{1}d\xi
_{2}.\nonumber
\end{align}
After taking $a+d-1=\beta=q$ and $b+c=\alpha=p$ in the left side of
(\ref{eq1}), using the finite biorthogonality relation (\ref{IKort}), we write
equation (\ref{eq1}) as%
\begin{align*}
& \frac{1}{2} \int\limits_{-\infty}^{\infty}\int\limits_{0}^{\infty}e^{-y}y^{a+d-1}%
\left(  1+x^{2}\right)  ^{-\left(  b+c-1/2\right)  }\ _{K}I_{k;\upsilon
}^{\left(  b+c,a+d-1\right)  }\left(  x,y\right)  \ _{K}\mathcal{I}%
_{s;\upsilon}^{\left(  c+b,d+a-1\right)  }\left(  x,y\right)  dxdy\\
&  =\frac{1}{2}\frac{k!2^{2\left(b+c-1\right)}\Gamma^{2}\left( b+c\right) \delta_{k,s}}{\left(b+c-k-1\right) \Gamma \left(  2\left( b+c\right) -k-1\right)  }\\
&  =\frac{\left(-1\right)^{k} 2^{2k-4}\left(\left(  1-b-c\right)  _{k}\right)^2}{\pi^{2}}
\int\limits_{-\infty}^{\infty}\int\limits_{-\infty}^{\infty}G_{1}\left(
a,b;\xi_{1},\xi_{2}\right)  \overline{G_{2}\left(  d,c;\xi_{1},\xi_{2}\right)
}\\
&  \times\Psi_{1}\left(  k,a,b,b+c,a+d-1,\upsilon;\xi_{1},\xi_{2}\right)
\overline{\Psi_{2}\left(  s,d,c,c+b,d+a-1,\upsilon;\xi_{1},\xi_{2}\right)
}d\xi_{1}d\xi_{2}.
\end{align*}
That is,%
\begin{align*}
&  \int\limits_{-\infty}^{\infty}\int\limits_{-\infty}^{\infty}G_{1}\left(
a,b;\xi_{1},\xi_{2}\right)  \Psi_{1}\left(  k,a,b,b+c,a+d-1,\upsilon;\xi
_{1},\xi_{2}\right) \\
&  \times\overline{G_{2}\left( d,c;\xi_{1},\xi_{2}\right)  }\overline
{\Psi_{2}\left(  s,d,c,c+b,d+a-1,\upsilon;\xi_{1},\xi_{2}\right)  }d\xi
_{1}d\xi_{2}\\
&  =\frac{\left(-1\right)^{k} 2^{2\left(  b+c-k\right)+1 } k!\pi^{2}\Gamma^2\left(  b+c-k\right)  }{\left(  b+c-k-1\right) \Gamma\left(  2\left(  b+c\right)
-k-1\right)  }\delta_{k,s}.
\end{align*}

Consequently, we give the theorem below.

\begin{theorem}
For $b>\left[ \frac{k}{2} \right]+\frac{1}{2}$, $c>\left[ \frac{s+1}{2} \right]+\frac{1}{2}$ and $a,d>0$, the
families $\Upsilon_{1}\left(  k,a,b,c,d,\upsilon;\xi_{1},\xi_{2}\right)$ and $\  \Upsilon_{2}\left(  s,d,c,b,a,\upsilon;\xi_{1},\xi_{2}\right)$ are
finite biorthogonal functions with respect to the weight function
\begin{align*}
\rho\left(
a,b,c,d;z,t\right)  =\Gamma\left(  a-it\right)  \Gamma\left(  d+it\right) \Gamma\left( \frac{1}{4}-\frac{iz}{2}\right) \Gamma\left( \frac{1}{4}+\frac{iz}{2}\right) \\
\times \Gamma\left(  b-\frac{1}{2}+\frac{iz}{2}\right)  \Gamma\left(  c-\frac{1}{2}-\frac{iz}{2}\right) 
\end{align*}
and the corresponding finite biorthogonality relation is%
\begin{align*}
&  \int\limits_{-\infty}^{\infty}\int\limits_{-\infty}^{\infty}\rho\left(
a,b,c,d;z,t\right) \Upsilon_{1}\left(  k,a,b,c,d,\upsilon
;iz,it\right)  \Upsilon_{2}\left(  s,d,c,b,a,\upsilon;-iz,-it\right)  dtdz\\
&  =\frac{\left(-1\right)^k 2^{2\left(  b+c-k\right)+1  }\pi^{2} k!  \Gamma^2 \left(  b+c-k\right) \Gamma\left(  b-1/4\right)
\Gamma\left(  c-1/4\right) \delta_{k,s}}{\left( b+c-k-1\right)\Gamma\left(  2\left(  b+c\right)  -k-1\right) },
\end{align*}
where both $k$ and $s$ are even or odd,
\[
\Upsilon_{1}\left(  k,a,b,c,d,\upsilon;z,t\right)  =2^{a-t} \Psi_{1}\left(  k,a,b,b+c,a+d-1,\upsilon
;-iz,-it\right)
\]
and%
\[
\Upsilon_{2}\left(  s,d,c,b,a,\upsilon;z,t\right)  =2^{d-t} \Psi_{2}\left(  s,d,c,c+b,d+a-1,\upsilon
;-iz,-it\right)  .
\]

\end{theorem}

\begin{remark}
The weight function of the finite biorthogonality relation in \textit{Theorem 24} is positive when $a=d$ and $b=c$.
\end{remark}

\textbf{Case 2:} Let $b>\left[ \frac{k}{2} \right]+\frac{1}{2}$, $c>\left[ \frac{s+1}{2} \right]+\frac{1}{2}$ and $s>k$.\\
\textbf{i)} When $k$ is odd and $s$ is even, replacing (\ref{F1}) and (\ref{F2}) in (\ref{Parseval}) and taking $e^{z}=x$ and $e^{t}=y$,
we obtain%
\begin{align}
&  \int\limits_{0}^{\infty}\int\limits_{0}^{\infty} x^{-1} e^{-y}y^{a+d-1}%
\left(  1+x^{2}\right)  ^{-\left(  b+c-1/2\right) } \ _{K}I_{k;\upsilon
}^{\left(  p,q\right)  }\left(  x,y\right)  \ _{K}\mathcal{I}_{s;\upsilon
}^{\left(  \alpha,\beta\right)  }\left(  x,y\right)  dxdy\label{eq1}\\
& =-2^{k+s-1} \left(1-b-c\right)_{k} \left(1-b-c\right)_{s} \ B\left( \frac{k+s}{2}, b+c-\frac{k+s+1}{2}\right) \nonumber \\
& \times F_{1,1,1}^{1,2,2}\left(
\genfrac{}{}{0pt}{}{b+c-\frac{k+s+1}{2}:\Delta\left(  2;-k\right);\Delta\left(  2;-s\right)  }{1-\frac{k+s}{2}:b+c-k;b+c-s};%
1;1\right) \nonumber \\
& =\frac{\left(  -1\right)  ^{k+s+\left[ \frac{k}{2} \right]+\left[ \frac{s+1}{2} \right]} 2^{k+s-4} \left(  1-p\right)  _{k} \left(  1-\alpha\right)  _{s}}{\pi^2}
\int\limits_{-\infty}^{\infty}\int\limits_{-\infty}^{\infty}G_{1}\left(
a,b;\xi_{1},\xi_{2}\right) \overline{G_{2}\left(  d,c;\xi_{1},\xi_{2}\right)  } \nonumber\\
&  \times\Psi_{1}\left(
k,a,b,p,q,\upsilon;\xi_{1},\xi_{2}\right)  \overline{\Psi_{2}\left(
s,d,c,\alpha,\beta,\upsilon;\xi_{1},\xi_{2}\right)  }d\xi_{1}d\xi
_{2}.\nonumber
\end{align} \\
\textbf{ii)} On the other hand, when $k$ is even and $s$ is odd, substituting (\ref{F1}) and (\ref{F2}) in (\ref{Parseval}) and using the transforms $e^{z}=x$ and $e^{t}=y$,
we have%
\begin{align}
&  \int\limits_{0}^{\infty}\int\limits_{0}^{\infty} x \ e^{-y}y^{a+d-1}%
\left(  1+x^{2}\right)  ^{-\left(  b+c-1/2\right) } \ _{K}I_{k;\upsilon
}^{\left(  p,q\right)  }\left(  x,y\right)  \ _{K}\mathcal{I}_{s;\upsilon
}^{\left(  \alpha,\beta\right)  }\left(  x,y\right)  dxdy\label{eq2}\\
& =-2^{k+s-1} \left(1-b-c\right)_{k} \left(1-b-c\right)_{s} \ B\left( \frac{k+s}{2}+1, b+c-\frac{k+s+1}{2} -1\right) \nonumber \\
& \times F_{1,1,1}^{1,2,2}\left(
\genfrac{}{}{0pt}{}{b+c-\frac{k+s+1}{2}-1:\Delta\left(  2;-k\right);\Delta\left(  2;-s\right)  }{-\frac{k+s}{2}:b+c-k;b+c-s};%
1;1\right) \nonumber \\
& =\frac{\left(-1\right)^{k+s+\left[ \frac{k}{2} \right]
+\left[ \frac{s+1}{2} \right]} 2^{k+s-4} \left(  1-p\right)  _{k} \left(  1-\alpha\right) _{s}}{\pi^2}
\int\limits_{-\infty}^{\infty}\int\limits_{-\infty}^{\infty} G_{1}\left(
a,b;\xi_{1},\xi_{2}\right) \overline{G_{2}\left(  d,c;\xi_{1},\xi_{2}\right)  } \nonumber\\
&  \times\Psi_{1}\left(
k,a,b,p,q,\upsilon;\xi_{1},\xi_{2}\right)  \overline{\Psi_{2}\left(
s,d,c,\alpha,\beta,\upsilon;\xi_{1},\xi_{2}\right)  }d\xi_{1}d\xi
_{2}.\nonumber
\end{align} \\
Combining (\ref{eq1}) and (\ref{eq2}), we get
\begin{align*}
& \int\limits_{-\infty}^{\infty}\int\limits_{-\infty}^{\infty} \overline{G_{2}\left(  d,c;\xi_{1},\xi_{2}\right)  }  \overline{\Psi_{2}\left(
s,d,c,\alpha,\beta,\upsilon;\xi_{1},\xi_{2}\right)  } \nonumber\\
&  \times G_{1}\left(
a,b;\xi_{1},\xi_{2}\right) \Psi_{1}\left(
k,a,b,p,q,\upsilon;\xi_{1},\xi_{2}\right) d\xi_{1}d\xi
_{2}\\
& = \left( -1\right) ^{k+1} 2^{3} \pi^2 \ B\left( \frac{k+s}{2}+\frac{1-\left(-1\right)^s}{2}, b+c-\frac{k+s+1}{2} -\frac{1-\left(-1\right)^s}{2}\right)  \\
& \times F_{1,1,1}^{1,2,2}\left(
\genfrac{}{}{0pt}{}{b+c-\frac{k+s+1}{2}-\frac{1-\left(-1\right)^s}{2}:\Delta\left(  2;-k\right);\Delta\left(  2;-s\right)  }{-\frac{k+s}{2}+\frac{1-\left(-1\right)^s}{2}:b+c-k;b+c-s};%
1;1\right)
\end{align*}
holds.
\bigskip
\subsection{A limit relation for I - Konhauser polynomials}

There exist the relationship%
\[
\lim_{p\rightarrow\infty}\left[  I_{k}^{\left(  p\right)  }\left(  \frac
{t}{\sqrt{p}}\right)  \ p^{-\frac{k}{2}}\right]  =H_{k}\left(  t\right)
\]
between the finite univariate orthogonal polynomials $I_{k}^{\left(  p\right)
}\left(  t\right)  $ and Hermite polynomials $H_{k}\left(  t\right)  $
\cite{GA}.

In this section, we give a similar relation for the pairs of I - Konhauser
polynomials and Hermite Konhauser polynomials \cite{OzEl} in the below results.

\begin{theorem}
The limit relations are satisfied between the I - Konhauser polynomials and
the 2D Hermite Konhauser polynomials:%
\begin{equation}
\lim_{p\rightarrow\infty}\left(  p^{-\frac{k}{2}}\ _{K}I_{k;\upsilon}^{\left(
p,q\right)  }\left(  \frac{z}{\sqrt{p}},t\right)  \right)  =\ _{\upsilon}H_{k}^{\left(  q\right)  }\left(  z,t\right)  \label{limI}%
\end{equation}
and%
\begin{equation}
\lim_{p\rightarrow\infty}\left(  p^{-\frac{s}{2}}\ _{K}\mathcal{I}%
_{s;\upsilon}^{\left(  p,q\right)  }\left(  \frac{z}{\sqrt{p}},t\right)
\right)  =\ _{\upsilon}Q_{s}^{\left(  q\right)  }\left(  z,t\right)
,\label{limQ}%
\end{equation}
where polynomials$\ _{\upsilon}H_{k}^{\left(  \beta\right)  }\left(
z,t\right)  $ and$\ _{\upsilon}Q_{s}^{\left(  \beta\right)  }\left(
z,t\right)  :=Q_{\upsilon}\left(  z,t\right)  $ are called the pair of the
Hermite Konhauser polynomials \cite{OzEl}, defined by (\ref{polH}) and
(\ref{HermiteQ}), respectively.
\end{theorem}

\begin{remark}
With the help of (\ref{limI}) and (\ref{limQ}), taking the limit of all
properties for the I - Konhauser polynomials the properties of the Hermite
Konhauser polynomials can be arrived.
\end{remark}
\bigskip
\section{The modified finite bivariate I - Konhauser polynomials}

In this section, using two new parameter $\gamma$ and $c$ we modify the I -
Konhauser polynomials$\ _{K}I_{k;\upsilon}^{\left(  p,q\right)  }\left(
z,t\right)  $ on the purpose of constructing fractional calculus operators
having the semigroup property. Also, we define the modified I - Konhauser
Mittag-Leffler functions corresponding to this modified polynomials to be
defined in \textit{Definition 28}.

\begin{definition}
The modified finite bivariate I - Konhauser polynomials are defined by the
representation%
\begin{equation}
_{K}I_{k;\upsilon}^{\left(  p,q;\gamma;c\right)  }\left(  z,t\right)  =\left(
1-p\right)  _{k}\left(  -2z\right)  ^{k}\sum_{l=0}^{\left[  \frac{k}%
{2}\right]  }\sum_{n=0}^{k-l}\frac{\left(  -1\right)  ^{l}\left(  -k\right)
_{2l}\left(  -\left(  k-l\right)  \right)  _{n}\left(  c\right)  _{l}\left(
\frac{1}{4z^{2}}\right)  ^{l}t^{\upsilon n}}{\left(  p-k\right)  _{l}\left(
\gamma\right)  _{l}\Gamma\left(  p+1+l\right)  \Gamma\left(  q+1+\upsilon
n\right)  l!n!} \label{MI}%
\end{equation}
for $\gamma\notin\mathbb{Z}^-\cup \{0\} $ and $p>\left\{  \max k\right\}  +1$, $q>-1$, $\upsilon
=1,2,...$ .
\end{definition}

\begin{remark}
For $c=0$, we have
\begin{equation*}
_{K}I_{k;\upsilon}^{\left(  p,q;\gamma;0\right)  }\left(  z,t\right)=\frac{\left(1-p\right)_{k} \left(-2z\right)^k}{\Gamma\left(p+1\right)} E_{\upsilon,q+1}^{\left(  -k\right)  }\left(  t^\upsilon\right),
\end{equation*}
and for $t=0$, we obtain
\begin{equation*}
_{K}I_{k;\upsilon}^{\left(  p,q;\gamma;c\right)  }\left(  z,0\right)=\frac{\left(1-p\right)_{k} \left(-2z\right)^k}{\Gamma\left(p+1\right)\Gamma\left(q+1\right)} \ _{3}F_{3}\left(
\genfrac{}{}{0pt}{}{-\frac{k}{2},\frac{-k+1}{2},c}{p-k,\gamma,p+1};%
\frac{-1}{z^{2}} \right).
\end{equation*}
Also,
\begin{equation*}
_{K}I_{k;\upsilon}^{\left(  p,q;\gamma;0\right)  }\left(  z,0\right)=\frac{\left(1-p\right)_{k} \left(-2z\right)^k}{\Gamma\left(p+1\right)\Gamma\left(q+1\right)}.
\end{equation*}
\end{remark}

\begin{theorem}
The representation%
\begin{equation}
_{K}I_{k;\upsilon}^{\left(  p,q;\gamma;c\right)  }\left(  z,t\right)
=\frac{\left(  1-p\right)  _{k}\left(  -2z\right)  ^{k}}{\Gamma\left(
p+1\right)  \Gamma\left(  q+1\right)  }F_{0,\upsilon,4}^{1,0,3}\left[
\genfrac{}{}{0pt}{}{-k:-;\Delta\left(  2;-k\right)  ,c;}{-:\Delta\left(
\upsilon;q+1\right)  ;-k,p-k,p+1,\gamma;}%
\frac{-1}{z^{2}};\left(  \frac{t}{\upsilon}\right)  ^{\upsilon}\right]
\label{MIkampe}%
\end{equation}
holds true for the modified\ I - Konhauser polynomials $_{K}I_{k;\upsilon
}^{\left(  p,q;\gamma;c\right)  }\left(  z,t\right)  .$
\end{theorem}

\begin{proof}
It can be completed in a similar vein to the proof for \textit{Theorem 9}.
\end{proof}

\begin{corollary}
The modified I - Konhauser polynomials $_{K}I_{k;\upsilon}^{\left(
p,q;\frac{-k+1}{2};p+1\right)  }\left(  z,t\right)  $ have the Kampe de
Feriet's double hypergeometric representation as follow%
\begin{equation}
_{K}I_{k;\upsilon}^{\left(  p,q;\frac{-\left(  k-1\right)  }{2};1+p\right)
}\left(  z,t\right)  =\frac{\left(  1-p\right)  _{k}\left(  -2z\right)  ^{k}%
}{\Gamma\left(  1+p\right)  \Gamma\left(  1+q\right)  }F_{0,\upsilon
,2}^{1,0,1}\left[
\genfrac{}{}{0pt}{}{-k:-;-k/2;}{-:\Delta\left(  \upsilon;q+1\right)  ;-k,p-k;}%
\frac{-1}{z^{2}};\left(  \frac{t}{\upsilon}\right)  ^{\upsilon}\right]  .
\label{Cor}%
\end{equation}

\end{corollary}

\begin{proof}
Choosing $\gamma=\frac{-k+1}{2}$ and $c=p+1$ in (\ref{MIkampe}), we obtain
(\ref{Cor}).
\end{proof}

\bigskip

Let us introduce the modified bivariate I - Konhauser Mittag-Leffler function
$E_{\gamma_{4}%
,\gamma_{5},\gamma_{6};p;q;\upsilon}^{\left(  \gamma_{1};\gamma_{2};\gamma_{3}\right)  }\left(  z,t\right)  $ corresponding to
polynomials $_{K}I_{k;\upsilon}^{\left(  p,q;\gamma;c\right)  }\left(
z,t\right)  $.

\begin{definition}
The definition of the modified bivariate I - Konhauser Mittag-Leffler function
$E_{\gamma_{4}%
,\gamma_{5},\gamma_{6};p;q;\upsilon}^{\left(  \gamma_{1};\gamma_{2};\gamma_{3}\right)  }\left(  z,t\right)  $ is given by%
\begin{equation}
E_{\gamma_{4}%
,\gamma_{5},\gamma_{6};p;q;\upsilon}^{\left(  \gamma_{1};\gamma_{2};\gamma_{3}\right)  }\left(  z,t\right)  =\sum_{l=0}^{\infty}%
\sum_{n=0}^{\infty}\frac{\left(  -1\right)  ^{l}\left(  \gamma_{1}\right)
_{2l}\left(  \gamma_{2}\right)  _{l+n}\left(  \gamma_{3}\right)  _{l}%
z^{l}t^{\upsilon n}}{\left(  \gamma_{4}\right)  _{l}\left(  \gamma_{5}\right)
_{l}\left(  \gamma_{6}\right)  _{l}\Gamma\left(  p+l\right)  \Gamma\left(
q+\upsilon n\right)  n!l!} \label{ModEdef}%
\end{equation}
for $p,q,\upsilon,\gamma_{1},\gamma_{2},\gamma_{3},\gamma_{4},\gamma
_{5},\gamma_{6}\in%
\mathbb{C}
$, $\operatorname{Re}\left(  p\right)  >0$, $\operatorname{Re}\left(
q\right)  >0$, $\operatorname{Re}\left(  \upsilon\right)  >0$ and $\gamma_{4},\gamma
_{5},\gamma_{6}\notin\mathbb{Z}^-\cup\{0\}$.
\end{definition}

\begin{remark}
In case $\gamma_{4}=$ $\frac{\gamma_{1}}{2}$, $\gamma_{5}=$ $\frac{\gamma
_{1}+1}{2}$, $\gamma_{6}=\gamma_{3}$ in (\ref{ModEdef}), the modified I - Konhauser Mittag-Leffler functions in \textit{Definition 31} are reduced to the Mittag-Leffler functions given in \cite{OzKurt}:%
\[
E_{\frac{\gamma_{1}%
}{2},\frac{\gamma_{1}+1}{2},\gamma_{3};p;q;\upsilon}^{\left(  \gamma_{1};\gamma_{2};\gamma_{3}\right)  }\left(  \frac{-z}{4},t\right)
=E_{p,q,\upsilon}^{\left(  \gamma_{2}\right)  }\left(  z,t\right),
\]
where $E_{p,q,\upsilon}^{\left(  \gamma_{2}\right)  }\left(  z,t\right)$ is given by (\ref{OKEdef}).
\end{remark}

\begin{remark}
Replacing $\gamma_{1}=0$ or $\gamma_{3}=0$ or $z=0$ provides us the following relation between the I - Konhauser Mittag-Leffler functions and the modified I - Konhauser Mittag-Leffler functions:
\[
E_{\gamma_{4},\gamma_{5},\gamma_{6};p;q;\upsilon}^{\left(  0;\gamma_{2};\gamma_{3}\right)  }\left(  z,t\right)
=E_{\gamma_{4},\gamma_{5},\gamma_{6};p;q;\upsilon}^{\left(  \gamma_{1};\gamma_{2};0\right)  }\left(  z,t\right)
=E_{\gamma_{4},\gamma_{5},\gamma_{6};p;q;\upsilon}^{\left(  \gamma_{1};\gamma_{2};\gamma_{3}\right)  }\left( 0,t\right)
=\frac{1}{\Gamma\left( p\right)}E_{\gamma_{3},\gamma_{4},q,\upsilon}^{\left( 0,\gamma_{2}\right)  }\left(  z,t\right),
\]
where $E_{\gamma_{3},\gamma_{4},q,\upsilon}^{\left( \gamma_{1},\gamma_{2}\right)  }\left(  z,t\right)$ is the I- Konhauser Mittag-Leffler functions given by (\ref{Edef}).\\
On the other hand, for $\gamma_{2}=0$,
\[
E_{\gamma_{4},\gamma_{5},\gamma_{6};p;q;\upsilon}^{\left(  \gamma_{1};0;\gamma_{3}\right)  }\left( z,t\right)
=\frac{1}{\Gamma\left( p\right)}E_{\gamma_{3},\gamma_{4};q;\upsilon}^{\left(  \gamma_{1};0\right)  }\left( z,t\right)
\]
holds.
\end{remark}

\begin{remark}
By choosing $\gamma_{4}=\frac{\gamma_{1}}{2},\gamma_{5}=\frac{\gamma_{1}+1}{2}, \gamma_{6}=\gamma_{2}$ and $t=0$, we have
\[
E_{\frac{\gamma_{1}}{2},\frac{\gamma_{1}+1}{2},\gamma_{2};p;q;\upsilon}^{\left(  \gamma_{1};\gamma_{2};\gamma_{3}\right)  }\left( z,0\right)
=\frac{1}{\Gamma\left( q\right)}E_{1,p}^{\left(  \gamma_{3}\right)  }\left( -4z\right),
\]
where $E_{1,p}^{\left(  \gamma_{3}\right)  }\left( -4z\right)$ is defined by (\ref{E2def}).
\end{remark}

\begin{corollary}
Comparing (\ref{MI}) and (\ref{ModEdef}) we have the relationship between the modified finite bivariate I - Konhauser polynomials and the modified I - Konhauser Mittag-Leffler functions as follow%
\begin{equation}
_{K}I_{k;\upsilon}^{\left(  p,q;\gamma;c\right)  }\left(  z,t\right)  =\left(
1-p\right)  _{k}\left(  -2z\right)  ^{k}E_{-k,p-k,\gamma;p+1;q+1;\upsilon}^{\left(
-k;-k;c\right)  }\left(  \frac{1}{4z^{2}},t\right)  .
\label{modIE}%
\end{equation}

\end{corollary}
\bigskip
\subsection{Integral and operational representation of the modified bivariate
I - Konhauser Mittag-Leffler functions and the modified finite bivariate I -
Konhauser polynomials}

In this subsection, we derive the integral and operational representations of
the modified I - Konhauser Mittag-Leffler functions $E_{\gamma_{4},\gamma_{5},\gamma_{6};p;q;\upsilon}^{\left(
\gamma_{1};\gamma_{2};\gamma_{3}\right)
}\left(  z,t\right)  $ and thus for the modified I - Konhauser polynomials
$_{K}I_{k;\upsilon}^{\left(  p,q;\gamma;c\right)  }\left(  z,t\right)  $.

\begin{theorem}
The modified I - Konhauser Mittag-Leffler functions $E_{\gamma_{4},\gamma_{5},\gamma_{6};p;q;\upsilon}^{\left(
\gamma_{1};\gamma_{2};\gamma_{3}\right)
}\left(  z,t\right)  $ have the following operational representation%
\begin{equation}
E_{\gamma_{4},\gamma_{5},\gamma_{6};p;q;\upsilon}^{\left(  \gamma_{1};\gamma_{2};\gamma_{3}\right)  }\left(  z,t\right)  =\frac{z^{1-p}t^{1-q}%
}{\left(  1-D_{t}^{-\upsilon}\right)  ^{\gamma_{2}}}\ _{4}F_{3}\left[
\genfrac{}{}{0pt}{}{\frac{\gamma_{1}}{2},\frac{\gamma_{1}+1}{2},\gamma
_{2},\gamma_{3}}{\gamma_{4},\gamma_{5},\gamma_{6}}%
;\frac{-4}{D_{z}\left(  1-D_{t}^{-\upsilon}\right)  }\right]  \left\{
\frac{z^{p-1}t^{q-1}}{\Gamma\left(  p\right)  \Gamma\left(  q\right)
}\right\}  . \label{Ehyp}%
\end{equation}

\end{theorem}

\begin{proof}
We can write $E_{\gamma_{4},\gamma_{5},\gamma_{6};p;q;\upsilon}^{\left(  \gamma_{1};\gamma_{2};\gamma
_{3}\right)  }\left(  z,t\right)  $ of the
form%
\begin{align*}
&  E_{\gamma_{4},\gamma_{5},\gamma_{6};p;q;\upsilon}^{\left(  \gamma_{1};\gamma_{2};\gamma
_{3}\right)  }\left(  z,t\right) \\
&  =\sum_{l=0}^{\infty}\sum_{n=0}^{\infty}\frac{\left(  -1\right)  ^{l}\left(
\gamma_{1}\right)  _{2l}\left(  \gamma_{2}\right)  _{l+n}\left(  \gamma
_{3}\right)  _{l}z^{1-p}t^{1-q}}{l!n!\left(  \gamma_{4}\right)  _{l}\left(
\gamma_{5}\right)  _{l}\left(  \gamma_{6}\right)  _{l}}D_{z}^{-l}%
D_{t}^{-\upsilon n}\left\{  \frac{z^{p-1}t^{q-1}}{\Gamma\left(  p\right)
\Gamma\left(  q\right)  }\right\} \\
&  =z^{1-p}t^{1-q}\sum_{l=0}^{\infty}\frac{\left(  -4\right)  ^{l}\left(
\frac{\gamma_{1}}{2}\right)  _{l}\left(  \frac{\gamma_{1}+1}{2}\right)
_{l}\left(  \gamma_{2}\right)  _{l}\left(  \gamma_{3}\right)  _{l}}{\left(
\gamma_{4}\right)  _{l}\left(  \gamma_{5}\right)  _{l}\left(  \gamma
_{6}\right)  _{l}l!}D_{z}^{-l}\sum_{n=0}^{\infty}\frac{\left(  \gamma
_{2}+l\right)  _{n}}{n!}D_{t}^{-\upsilon n}\left\{  \frac{z^{p-1}t^{q-1}%
}{\Gamma\left(  p\right)  \Gamma\left(  q\right)  }\right\} \\
&  =\frac{z^{1-p}t^{1-q}}{\left(  1-D_{t}^{-\upsilon}\right)  ^{\gamma_{2}}%
}\sum_{l=0}^{\infty}\frac{\left(  \frac{\gamma_{1}}{2}\right)  _{l}\left(
\frac{\gamma_{1}+1}{2}\right)  _{l}\left(  \gamma_{2}\right)  _{l}\left(
\gamma_{3}\right)  _{l}}{\left(  \gamma_{4}\right)  _{l}\left(  \gamma
_{5}\right)  _{l}\left(  \gamma_{6}\right)  _{l}l!}\left(  \frac{-4}%
{D_{z}\left(  1-D_{t}^{-\upsilon}\right)  }\right)  ^{l}\left\{  \frac
{z^{p-1}t^{q-1}}{\Gamma\left(  p\right)  \Gamma\left(  q\right)  }\right\} \\
&  =\frac{z^{1-p}t^{1-q}}{\left(  1-D_{t}^{-\upsilon}\right)  ^{\gamma_{2}}%
}\ _{4}F_{3}\left[
\genfrac{}{}{0pt}{}{\frac{\gamma_{1}}{2},\frac{\gamma_{1}+1}{2},\gamma
_{2},\gamma_{3}}{\gamma_{4},\gamma_{5},\gamma_{6}}%
;\frac{-4}{D_{z}\left(  1-D_{t}^{-\upsilon}\right)  }\right]  \left\{
\frac{z^{p-1}t^{q-1}}{\Gamma\left(  p\right)  \Gamma\left(  q\right)
}\right\}  .
\end{align*}

\end{proof}

\begin{theorem}
The modified I - Konhauser polynomials $_{K}I_{k;\upsilon}^{\left(
p,q;\gamma;c\right)  }\left(  z,t\right)  $ may be stated as follow%
\begin{align}
_{K}I_{k;\upsilon}^{\left(  p,q;\gamma;c\right)  }\left(  z,t\right)   &
=\left(  -1\right)  ^{k}\left(  1-p\right)  _{k}\left(  2z\right)
^{k+2p}t^{-q}\left(  1-D_{t}^{-\upsilon}\right)  ^{k}\label{Ihyp}\\
&  \times\ _{3}F_{2}\left[
\genfrac{}{}{0pt}{}{\frac{-k}{2},\frac{-k+1}{2},c}{-k+p,\gamma}%
;\frac{2}{z^{3}D_{z}\left(  1-D_{t}^{-\upsilon}\right)  }\right]  \left\{
\frac{\left(  2z\right)  ^{-2p}t^{q}}{\Gamma\left(  p+1\right)  \Gamma\left(
q+1\right)  }\right\}  .\nonumber
\end{align}

\end{theorem}

\begin{proof}
(\ref{Ihyp}) follows from (\ref{Ehyp}) and (\ref{modIE}).
\end{proof}

\begin{corollary}
The following operational representations are satisfied for $_{K}%
I_{k;\upsilon}^{\left(  p,q;\gamma;c\right)  }\left(  z,t\right)  :$%
\begin{align*}
_{K}I_{k;\upsilon}^{\left(  p,q;\frac{-k+1}{2};p-k\right)  }\left(
z,t\right)   &  =\frac{\Gamma\left(  p\right)  }{\Gamma\left(  p-k\right)
}\left(  2z\right)  ^{k+2p}t^{-q}\left(  1-D_{t}^{-\upsilon}\right)  ^{k}\\
&  \times\left(  1-\frac{2}{z^{3}D_{z}\left(  1-D_{t}^{-\upsilon}\right)
}\right)  ^{\frac{k}{2}}\left\{  \frac{\left(  2z\right)  ^{-2p}t^{q}}%
{\Gamma\left(  1+q\right)  \Gamma\left(  1+p\right)  }\right\}  ,
\end{align*}%
\begin{align*}
_{K}I_{k;\upsilon}^{\left(  p,q;p+1;p-k\right)  }\left(  z,t\right)   &
=\frac{\Gamma\left(  p\right)  }{\Gamma\left(  p-k\right)  }\left(  2z\right)
^{k+2p}t^{-q}\left(  1-D_{t}^{-\upsilon}\right)  ^{k}\\
&  \times\ _{2}F_{1}\left[
\genfrac{}{}{0pt}{}{\frac{-k}{2},\frac{-k+1}{2}}{p+1}%
;\frac{2}{z^{3}D_{z}\left(  1-D_{t}^{-\upsilon}\right)  }\right]  \left\{
\frac{\left(  2z\right)  ^{-2p}t^{q}}{\Gamma\left(  1+q\right)  \Gamma\left(
1+p\right)  }\right\},
\end{align*}
\begin{align*}
_{K}I_{k;\upsilon}^{\left(  p,q;\gamma;0\right)  }\left(  z,t\right)   &
=\left( -1\right)^k \left( 1-p\right)_k \left(  2z\right)
^{k+2p}t^{-q}\left(  1-D_{t}^{-\upsilon}\right)  ^{k}
 \left\{  \frac{\left(
2z\right)  ^{-2p}t^{q}}{\Gamma\left(  1+q\right)  \Gamma\left(  1+p\right)
}\right\}
\end{align*}
and
\begin{align*}
_{K}I_{k;\upsilon}^{\left(  p,q;\gamma;\gamma\right)  }\left(  z,t\right)   &
=\frac{\Gamma\left(  p\right)  }{\Gamma\left(  p-k\right)  }\left(  2z\right)
^{k+2p}t^{-q}\left(  1-D_{t}^{-\upsilon}\right)  ^{k}\\
&  \times\ _{2}F_{1}\left[
\genfrac{}{}{0pt}{}{\frac{-k}{2},\frac{-k+1}{2}}{p-k}%
;\frac{2}{z^{3}D_{z}\left(  1-D_{t}^{-\upsilon}\right)  }\right]  \left\{
\frac{\left(  2z\right)  ^{-2p}t^{q}}{\Gamma\left(  1+q\right)  \Gamma\left(
1+p\right)  }\right\}.
\end{align*}
\end{corollary}

\begin{theorem}
For parameters $\gamma_{1},\gamma_{3},\gamma_{4},\gamma_{5},\gamma_{6}$ with positive real parts, the modified I - Konhauser Mittag-Leffler functions have the integral
representation as follow%
\begin{align*}
E_{\gamma_{4},\gamma_{5},\gamma_{6};p;q;\upsilon}^{\left(  \gamma_{1};\gamma_{2};\gamma
_{3}\right)  }\left(  z,t\right)   &  =\frac{\Gamma\left(
\gamma_{4}\right)  \Gamma\left(  \gamma_{5}\right)  \Gamma\left(  \gamma_{6}\right)  }{i\left(  2\pi\right)  ^{5}\Gamma\left(  \gamma_{1}\right) \Gamma\left(  \gamma_{3}\right)  }\int\limits_{0}^{\infty}\int\limits_{0}%
^{\infty}\int\limits_{-\infty}^{0^{+}}\int\limits_{-\infty}^{0^{+}} \int\limits_{-\infty}^{0^{+}}\int\limits_{-\infty}^{0^{+}}\int\limits_{-\infty
}^{0^{+}}u_{1}^{\gamma_{1}-1}u_{2}^{\gamma_{3}-1}\label{modEintrep}\\
&  \times w_{1}^{-\gamma_{4}}w_{2}^{-\gamma_{5}}w_{3}^{-\gamma_{6}}w_{4}^{-p}w_{5}^{-q}e^{-u_{1}-u_{2}+w_{1}+w_{2}+w_{3}+w_{4}+w_{5}}\nonumber\\
&  \times\left(  \frac{w_{5}^{\upsilon}-t^{\upsilon}}{w_{5}^{\upsilon}}%
+\frac{u_{1}^{2}u_{2}z}{w_{1}w_{2}w_{3}w_{4}}\right)  ^{-\gamma_{2}}%
dw_{5}dw_{4}dw_{3}dw_{2}dw_{1}du_{2}du_{1}.\nonumber
\end{align*}

\end{theorem}

\begin{proof}
The proof follows from (\ref{Gamma}) and (\ref{PGamma}).
\end{proof}

\begin{theorem}
For $\operatorname{Re\left(\gamma\right)}>0$ and $\operatorname{Re\left(c\right)}>0$, the integral representation of modified\ I - Konhauser polynomials
$_{K}I_{k;\upsilon}^{\left(  p,q;\gamma;c\right)  }\left(  z,t\right)  $ is as
follow
\begin{align*}
_{K}I_{k;\upsilon}^{\left(  p,q;\gamma;c\right)  }\left(  z,t\right)   &
=\frac{z^{k} \sqrt{\pi} \Gamma\left(
p\right)  \Gamma\left(  \gamma\right)}{- \left(2\pi\right)^{6} \Gamma\left(  c\right)}\int\limits_{0}^{\infty}\int\limits_{0}^{\infty}%
\int\limits_{-\infty}^{0^{+}} \int\limits_{-\infty}^{0^{+}} \int\limits_{-\infty}^{0^{+}}\int\limits_{-\infty
}^{0^{+}}\int\limits_{-\infty}^{0^{+}}\int\limits_{-\infty}^{0^{+}}%
u_{1}^{c-1}u_{2}^{k}\\
&  \times w_{1}^{-p-1}w_{2}^{-q-1}w_{3}^{-\gamma}w_{4}^{k-p}w_{5}^{-\frac{k}{2}-1}w_{6}^{-\frac{k+1}{2}}
e^{-u_{1}-u_{2}+w_{1}+w_{2}+w_{3}+w_{4}+w_{5}+w_{6}}\\
&  \times\left(  \frac{w_{2}^{\upsilon}-t^{\upsilon}}{w_{2}^{\upsilon}}%
-\frac{u_{1}w_{5}w_{6}}{u_{2}w_{1}w_{3}w_{4}z^{2}}\right)  ^{k}dw_{6}dw_{5}
dw_{4}dw_{3}dw_{2}dw_{1}du_{2}du_{1}.
\end{align*}
\end{theorem}

\begin{proof}
Taking account of (\ref{Gamma}) and (\ref{PGamma}), we complete the proof.
\end{proof}

\begin{corollary}
For the modified polynomials $_{K}I_{k;\upsilon}^{\left(  p,q;\gamma;c\right)
}\left(  z,t\right)  $, we have the following integral representation%
\begin{align*}
_{K}I_{k;\upsilon}^{\left(  p,q;\gamma;p\right)  }\left(  z,t\right)   &
=\frac{z^{k} \sqrt{\pi} \Gamma\left(  \gamma\right) }{- \left(2\pi\right)  ^{6}}\int\limits_{0}^{\infty}\int\limits_{0}^{\infty}%
\int\limits_{-\infty}^{0^{+}} \int\limits_{-\infty}^{0^{+}} \int\limits_{-\infty}^{0^{+}}\int\limits_{-\infty
}^{0^{+}}\int\limits_{-\infty}^{0^{+}}\int\limits_{-\infty}^{0^{+}}%
u_{1}^{p-1}u_{2}^{k}\\
&  \times w_{1}^{-p-1}w_{2}^{-q-1}w_{3}^{-\gamma}w_{4}^{k-p}w_{5}^{-\frac{k}{2}-1}w_{6}^{-\frac{k+1}{2}}
e^{-u_{1}-u_{2}+w_{1}+w_{2}+w_{3}+w_{4}+w_{5}+w_{6}}\\
&  \times\left(  \frac{w_{2}^{\upsilon}-t^{\upsilon}}{w_{2}^{\upsilon}}%
-\frac{u_{1}w_{5}w_{6}}{u_{2}w_{1}w_{3}w_{4}z^{2}}\right)  ^{k}dw_{6}dw_{5}
dw_{4}dw_{3}dw_{2}dw_{1}du_{2}du_{1},
\end{align*}
where $\operatorname{Re\left(\gamma\right)}>0$.
\end{corollary}

\subsection{Some properties for the modified I - Konhauser polynomials and the
modified I - Konhauser Mittag-Leffler functions}

In this section, we compute the Laplace transform and fractional integral and
derivative operators\ of $E_{\gamma_{4},\gamma_{5},\gamma_{6};p;q;\upsilon}^{\left(  \gamma_{1};\gamma
_{2};\gamma_{3}\right)  }\left(  z,t\right)
$ and $_{K}I_{k;\upsilon}^{\left(  p,q;\gamma;c\right)  }\left(  z,t\right)  $.

\begin{theorem}
By choosing $\gamma_{4}=\frac{\gamma_{1}+1}{2},\gamma_{5}=\gamma_{2}$ and
$\gamma_{6}=\gamma_{3}$,%
\[
L\left\{  t^{q}E_{\frac{\gamma_{1}+1}{2},\gamma_{2},\gamma_{3};p+1;q+1;\upsilon}^{\left(  \gamma_{1};\gamma_{2};\gamma
_{3}\right)  }\left(  z,wt\right)
\right\}  =\frac{1}{a^{q+1}}\left(  \frac{a^{\upsilon}-w^{\upsilon}%
}{a^{\upsilon}}\right)  ^{-\gamma_{2}}E_{1,p+1}^{\left(  \frac{\gamma_{1}%
}{2}\right)  }\left(  \frac{-4za^{\upsilon}}{a^{\upsilon}-w^{\upsilon}%
}\right)
\]
holds, where $\left\vert \frac{w^{\upsilon}}{a^{\upsilon}}\right\vert <1$, and
$E_{p,q}^{\left(  \gamma\right)  }\left(  t\right)  $ is defined in
\cite{Prab}.
\end{theorem}

\begin{proof}%
\begin{align*}
L\left\{  t^{q}E_{\frac{\gamma_{1}+1}{2},\gamma_{2},\gamma_{3};p+1;q+1;\upsilon}^{\left(  \gamma_{1};\gamma_{2};\gamma
_{3}\right)  }\left(  z,wt\right)
\right\}   &  =\frac{1}{a^{q+1}}\sum_{l=0}^{\infty}\sum_{n=0}^{\infty}%
\frac{\left(  \frac{\gamma_{1}}{2}\right)  _{l}\left(  \gamma_{2}+l\right)
_{n}\left(  -4z\right)  ^{l}\left(  \frac{w^{\upsilon}}{a^{\upsilon}}\right)
^{n}}{n!l!\Gamma\left(  p+l+1\right)  }\\
&  =\frac{1}{a^{q+1}}\left(  1-\frac{w^{\upsilon}}{a^{\upsilon}}\right)
^{-\gamma_{2}}\sum_{l=0}^{\infty}\frac{\left(  \frac{\gamma_{1}}{2}\right)
_{l}}{\Gamma\left(  p+l+1\right)  l!}\left(  \frac{-4za^{\upsilon}%
}{a^{\upsilon}-w^{\upsilon}}\right)  ^{l}\\
&  =\frac{1}{a^{q+1}}\left(  1-\frac{w^{\upsilon}}{a^{\upsilon}}\right)
^{-\gamma_{2}}E_{1,p+1}^{\left(  \frac{\gamma_{1}}{2}\right)  }\left(
\frac{-4za^{\upsilon}}{a^{\upsilon}-w^{\upsilon}}\right)  .
\end{align*}

\end{proof}

The bivariate Laplace transform \cite{KG} is defined of form%
\[
L_{2}\left[  h\left(  z,t\right)  \right]  \left(  \zeta,\kappa\right)
=\int\limits_{0}^{\infty}\int\limits_{0}^{\infty}h\left(  z,t\right)
e^{-\left(  \zeta z+\kappa t\right)  }dzdt
\]
for $\operatorname{Re}\left(  \zeta\right)  >0,\operatorname{Re}\left(
\kappa\right)  >0.$

\begin{theorem}
For $\left\vert \frac{w_{2}^{\upsilon}}{b^{\upsilon}}\right\vert <1$ and
$\left\vert \frac{4w_{1}b^{\upsilon}}{a\left(  w_{2}^{\upsilon}-b^{\upsilon
}\right)  }\right\vert <1$, the Laplace transform%
\[
L_{2}\left\{  z^{p}t^{q}E_{\frac{\gamma_{1}+1}{2},\gamma_{2},\gamma_{3};p+1;q+1;\upsilon}^{\left(  \gamma_{1};\gamma_{2}%
;\gamma_{3}\right)  }\left(
w_{1}z,w_{2}t\right)  \right\}  =\frac{1}{a^{p+1}b^{q+1}}\left(
\frac{b^{\upsilon}-w_{2}^{\upsilon}}{b^{\upsilon}}\right)  ^{-\gamma_{2}%
}\left(  1-\frac{4w_{1}b^{\upsilon}}{a\left(  w_{2}^{\upsilon}-b^{\upsilon
}\right)  }\right)  ^{-\frac{\gamma_{1}}{2}}%
\]
holds for $E_{\gamma_{4},\gamma_{5},\gamma_{6};p;q;\upsilon}^{\left(  \gamma_{1};\gamma_{2};\gamma_{3}\right)  }\left(  z,t\right)  $.
\end{theorem}

\begin{proof}%
\begin{align*}
&  L_{2}\left\{  z^{p}t^{q}E_{\frac{\gamma_{1}+1}{2},\gamma_{2},\gamma_{3};p+1;q+1;\upsilon}^{\left(  \gamma_{1};\gamma
_{2};\gamma_{3}\right)  }\left(
w_{1}z,w_{2}t\right)  \right\} \\
&  =\sum_{l=0}^{\infty}\sum_{n=0}^{\infty}\frac{\left(  \frac{\gamma_{1}}%
{2}\right)  _{l}\left(  \gamma_{2}+l\right)  _{n}\left(  -4w_{1}\right)
^{l}w_{2}^{\upsilon n}}{n!l!\Gamma\left(  p+l+1\right)  \Gamma\left(
q+\upsilon n+1\right)  }\int\limits_{0}^{\infty}\int\limits_{0}^{\infty
}e^{-\left(  az+bt\right)  }z^{p+l}t^{q+\upsilon n}dzdt\\
&  =\frac{1}{a^{p+1}b^{q+1}}\sum_{l=0}^{\infty}\sum_{n=0}^{\infty}\left(
\gamma_{2}+l\right)  _{n}\left(  \frac{\gamma_{1}}{2}\right)  _{l}%
\frac{\left(  \frac{w_{2}^{\upsilon}}{b^{\upsilon}}\right)  ^{n}\left(
\frac{-4w_{1}}{a}\right)  ^{l}}{n!l!}\\
&  =\frac{1}{a^{p+1}b^{q+1}}\left(  1-\frac{w_{2}^{\upsilon}}{b^{\upsilon}%
}\right)  ^{-\gamma_{2}}\sum_{l=0}^{\infty}\frac{\left(  \frac{\gamma_{1}%
}{2}\right)  _{l}}{l!}\left(  \frac{-4w_{1}b^{\upsilon}}{a\left(  b^{\upsilon
}-w_{2}^{\upsilon}\right)  }\right)  ^{l}\\
&  =\frac{1}{a^{p+1}b^{q+1}}\left(  1-\frac{w_{2}^{\upsilon}}{b^{\upsilon}%
}\right)  ^{-\gamma_{2}}\left(  1-\frac{4w_{1}b^{\upsilon}}{a\left(
w_{2}^{\upsilon}-b^{\upsilon}\right)  }\right)  ^{-\frac{\gamma_{1}}{2}}.
\end{align*}

\end{proof}

\begin{corollary}
By choosing $\gamma=\frac{-k+1}{2}$ and $c=p-k$ in (\ref{modIE}), the Laplace transform%
\begin{align*}
&  L_{2}\left\{  w_{1}^{\frac{k}{2}}z^{p+\frac{k}{2}}t^{q}\ _{K}I_{k;\upsilon
}^{\left(  p,q;\frac{-k+1}{2};p-k\right)  }\left(  \frac{1}{2\sqrt{w_{1}z}%
},w_{2}t\right)  \right\} \\
&  =\frac{\left(  1-p\right)  _{k}}{a^{p+1}b^{q+1}}\left(  \frac
{w_{2}^{\upsilon}-b^{\upsilon}}{b^{\upsilon}}\right)  ^{k}\left(
1-\frac{4w_{1}b^{\upsilon}}{a\left(  w_{2}^{\upsilon}-b^{\upsilon}\right)
}\right)  ^{\frac{k}{2}}%
\end{align*}
holds for $\left\vert \frac{w_{2}^{\upsilon}}{b^{\upsilon}}\right\vert <1$ and
$\left\vert \frac{4w_{1}b^{\upsilon}}{a\left(  w_{2}^{\upsilon}-b^{\upsilon
}\right)  }\right\vert <1$.
\end{corollary}

\bigskip

For a function $d\left(  z,t\right)  $,\ the double fractional integral is
defined as follow \cite{Kilbas}%
\[
\left(  _{t}\mathbb{I}_{b^{+}}^{\tau}\ _{z}\mathbb{I}_{a^{+}}^{\sigma}\right)
d\left(  z,t\right)  =\frac{1}{\Gamma\left(  \tau\right)  \Gamma\left(
\sigma\right)  }\int\limits_{b}^{t}\int\limits_{a}^{z}\left(  z-\xi\right)
^{\sigma-1}\left(  t-y\right)  ^{\tau-1}d\left(  \xi,y\right)  d\xi dy
\]
for $z>a,\ t>b$ and $\operatorname{Re}\left(  \tau\right)
>0,\ \operatorname{Re}\left(  \sigma\right)  >0$.

For a function $d\left(  z,t\right)  $, the double fractional derivative is
defined as follow \cite{Kilbas}%
\[
\left(  \ _{z}D_{a^{+}}^{\sigma}\ _{t}D_{b^{+}}^{\tau}\right)  d\left(
z,t\right)  =\left(  \frac{\partial}{\partial t}\right)  ^{s}\left(
\frac{\partial}{\partial z}\right)  ^{k}\left(  \ _{z}\mathbb{I}_{a^{+}%
}^{k-\sigma}\ _{t}\mathbb{I}_{b^{+}}^{s-\tau}\right)  d\left(  z,t\right)  ,
\]
for $z>a,\ t>b$ and$\ s=\left[  \operatorname{Re}\left(  \tau\right)  \right]
+1,\ k=\left[  \operatorname{Re}\left(  \sigma\right)  \right]  +1$.

\begin{theorem}
For $\operatorname{Re}\left(  \tau\right)
>0$, the modified bivariate I - Konhauser Mittag-Leffler functions have the
following Riemann-Liouville fractional integral representation:%
\begin{align}
&  \left(  \ _{t}\mathbb{I}_{b^{+}}^{\tau}\right)  \left[  \left(  t-b\right)
^{q-1}E_{\gamma
_{4},\gamma_{5},\gamma_{6};p;q;\upsilon}^{\left(  \gamma_{1};\gamma_{2};\gamma_{3}\right)  }\left(  z,w\left(  t-b\right)  \right)
\right]  =\left(  t-b\right)  ^{q+\tau-1}E_{\gamma_{4},\gamma_{5},\gamma_{6};p;q+\tau;\upsilon}^{\left(  \gamma
_{1};\gamma_{2};\gamma_{3}\right)  }\left(
z,w\left(  t-b\right)  \right).\label{modEfracint}
\end{align}

\end{theorem}

\begin{proof}
This can be done in a similar way to the proof of \textit{Theorem 20}.
\end{proof}

\begin{corollary}
The modified I - Konhauser polynomials$\ _{K}I_{k;\upsilon}^{\left(
p,q;\gamma;c\right)  }\left(  z,t\right)  $ have the representation of the
Riemann-Liouville fractional integral as follow%
\begin{align*}
&  _{t}\mathbb{I}_{b^{+}}^{\tau}\left(  \left(  t-b\right)  ^{q}%
\ _{K}I_{k;\upsilon}^{\left(  p,q;\gamma;c\right)  }\left(  z,w\left(
t-b\right)  \right)  \right) =\left(  t-b\right)  ^{q+\tau}\ _{K}I_{k;\upsilon}^{\left(  p,q+\tau
;\gamma;c\right)  }\left(  z,w\left(  t-b\right)  \right)
,\ \operatorname{Re}\left(  \tau\right)  >0.
\end{align*}

\end{corollary}

\begin{proof}
The proof follows from (\ref{modIE}) and (\ref{modEfracint}).
\end{proof}

\begin{theorem}
For $\operatorname{Re}\left(  \sigma\right)  >0$ and$\ \operatorname{Re}%
\left(  \tau\right)  >0$, the double fractional integral representation of the
modified bivariate I - Konhauser Mittag-Leffler functions is as follows%
\begin{align}
&  \left(  \ _{z}\mathbb{I}_{a^{+}}^{\sigma}\ _{t}\mathbb{I}_{b^{+}}^{\tau
}\right)  \left[  \left(  z-a\right)  ^{p-1}\left(  t-b\right)  ^{q-1}%
E_{\gamma_{4}%
,\gamma_{5},\gamma_{6};p;q;\upsilon}^{\left(  \gamma_{1};\gamma_{2};\gamma_{3}\right)  }\left(  w_{1}\left(  z-a\right)  ,w_{2}\left(
t-b\right)  \right)  \right] \label{modEdoublefracint}\\
&  =\left(  z-a\right)  ^{p+\sigma-1}\left(  t-b\right)  ^{q+\tau
-1}E_{\gamma_{4},\gamma_{5},\gamma_{6};p+\sigma;q+\tau;\upsilon}^{\left(  \gamma_{1};\gamma_{2};\gamma
_{3}\right)  }\left(  w_{1}\left(
z-a\right)  ,w_{2}\left(  t-b\right)  \right)  .\nonumber
\end{align}

\end{theorem}

\begin{proof}%
\begin{align*}
&  \left(  \ _{z}\mathbb{I}_{a^{+}}^{\sigma}\ _{t}\mathbb{I}_{b^{+}}^{\tau
}\right)  \left[  \left(  z-a\right)  ^{p-1}\left(  t-b\right)  ^{q-1}%
E_{\gamma_{4},\gamma_{5},\gamma_{6};p;q;\upsilon}^{\left(  \gamma_{1};\gamma_{2};\gamma_{3}\right)  }\left(  w_{1}\left(  z-a\right)  ,w_{2}\left(
t-b\right)  \right)  \right] \\
&  =\frac{1}{\Gamma\left(  \tau\right)  \Gamma\left(  \sigma\right)  }%
\sum\limits_{l=0}^{\infty}\sum\limits_{n=0}^{\infty}\frac{\left(  -1\right)
^{l}\left(  \gamma_{1}\right)  _{2l}\left(  \gamma_{2}\right)  _{l+n}\left(
\gamma_{3}\right)  _{l}w_{1}^{l}w_{2}^{\upsilon n}}{\left(  \gamma_{4}\right)
_{l}\left(  \gamma_{5}\right)  _{l}\left(  \gamma_{6}\right)  _{l}%
\Gamma\left(  p+l\right)  \Gamma\left(  q+\upsilon n\right)  l!n!}\\
&  \times\int\limits_{a}^{z}\int\limits_{b}^{t}\left(  z-\xi\right)
^{\sigma-1}\left(  \xi-a\right)  ^{p+l-1}\left(  t-y\right)  ^{\tau-1}\left(
y-b\right)  ^{q+\upsilon n-1}dyd\xi\\
&  =\left(  z-a\right)  ^{p+\sigma-1}\left(  t-b\right)  ^{q+\tau-1}%
\sum\limits_{l=0}^{\infty}\sum\limits_{n=0}^{\infty}\frac{\left(  \gamma
_{1}\right)  _{2l}\left(  \gamma_{2}\right)  _{l+n}\left(  \gamma_{3}\right)
_{l}\left(  -w_{1}\left(  z-a\right)  \right)  ^{l}\left(  w_{2}\left(
t-b\right)  \right)  ^{\upsilon n}}{\left(  \gamma_{4}\right)  _{l}\left(
\gamma_{5}\right)  _{l}\left(  \gamma_{6}\right)  _{l}\Gamma\left(
p+\sigma+l\right)  \Gamma\left(  q+\tau+\upsilon n\right)  l!n!}\\
&  =\left(  z-a\right)  ^{p+\sigma-1}\left(  t-b\right)  ^{q+\tau
-1}E_{\gamma_{4},\gamma_{5},\gamma_{6};p+\sigma;q+\tau;\upsilon}^{\left(  \gamma_{1};\gamma_{2};\gamma
_{3}\right)  }\left(  w_{1}\left(
z-a\right)  ,w_{2}\left(  t-b\right)  \right)  .
\end{align*}

\end{proof}

\begin{corollary}
For $\operatorname{Re}\left(  \sigma\right)  >0$ and$\ \operatorname{Re}%
\left(  \tau\right)  >0$, the double fractional integral representation of the
polynomials$\ _{K}I_{k;\upsilon}^{\left(  p,q;\gamma;c\right)  }\left(
z,t\right)  $ is as follows
\begin{align*}
&  \left(  \ _{z}\mathbb{I}_{a^{+}}^{\sigma}\ _{t}\mathbb{I}_{b^{+}}^{\tau
}\right)  \left[  \left(  z-a\right)  ^{p}\left(  t-b\right)  ^{q}\left(
\frac{-i}{\sqrt{w_{1}\left(  z-a\right)  }}\right)  ^{-k}\ _{K}I_{k;\upsilon
}^{\left(  p,q;\gamma;c\right)  }\left(  \frac{i}{2\sqrt{w_{1}\left(
z-a\right)  }},w_{2}\left(  t-b\right)  \right)  \right] \\
&  =\left(  z-a\right)  ^{p+\sigma}\left(  t-b\right)  ^{q+\tau}\left(
\frac{-i}{\sqrt{w_{1}\left(  z-a\right)  }}\right)  ^{-k}\ _{K}I_{k;\upsilon
}^{\left(  p+\sigma,q+\tau;\gamma;c\right)  }\left(  \frac{i}{2\sqrt
{w_{1}\left(  z-a\right)  }},w_{2}\left(  t-b\right)  \right)  .
\end{align*}

\end{corollary}

\begin{proof}
(\ref{modEdoublefracint}) and (\ref{modIE}) are used.
\end{proof}

\begin{theorem}
For $\operatorname{Re}\left(  \tau\right)  \geq0$ and $\operatorname{Re}\left(  q-\tau\right)  \geq0$, the modified I - Konhauser
Mittag-Leffler functions have the Riemann-Liouville partial fractional
derivative representation:%
\begin{equation}
\ _{t}D_{b^{+}}^{\tau}\left[  \left(  t-b\right)  ^{q-1}E_{\gamma_{4},\gamma_{5},\gamma
_{6};p;q;\upsilon
}^{\left(  \gamma_{1};\gamma_{2};\gamma_{3}\right)  }\left(  z,w\left(  t-b\right)  \right)  \right]  =\left(
t-b\right)  ^{q-\tau-1}E_{\gamma_{4},\gamma_{5},\gamma_{6};p;q-\tau;\upsilon}^{\left(  \gamma_{1};\gamma
_{2};\gamma_{3}\right)  }\left(  z,w\left(
t-b\right)  \right)  . \label{modEfracder}%
\end{equation}

\end{theorem}

\begin{proof}
It is in a similar vein to the proof of (\ref{Efracder}).
\end{proof}

\begin{corollary}
For $\operatorname{Re}\left(  \tau\right)  \geq0$ and $\operatorname{Re}\left(  q-\tau\right)  \geq0$, the modified I - Konhauser polynomials have the Riemann-Liouville fractional
derivative representation as follow%
\[
\ _{t}D_{b^{+}}^{\tau}\left[  \left(  t-b\right)  ^{q}\ _{K}I_{k;\upsilon
}^{\left(  p,q;\gamma;c\right)  }\left(  z,w\left(  t-b\right)  \right)
\right]  =\left(  t-b\right)  ^{q-\tau}\ _{K}I_{k;\upsilon}^{\left(
p,q-\tau;\gamma;c\right)  }\left(  z,w\left(  t-b\right)  \right).
\]

\end{corollary}

\begin{proof}
It is proved by using relationship (\ref{modIE}) in (\ref{modEfracder}).\
\end{proof}

\begin{theorem}
For $\ \operatorname{Re}%
\left(  \sigma\right)  \geq0$, $\ \operatorname{Re}%
\left(  p-\sigma\right)  \geq0$ and $\operatorname{Re}\left(  \tau\right)  \geq0$, $\operatorname{Re}\left(  q-\tau\right)  \geq0$, the modified I - Konhauser Mittag-Leffler
functions have the following double partial fractional derivative
representation:%
\begin{align}
&  \left(  \ _{z}D_{a^{+}}^{\sigma}\ _{t}D_{b^{+}}^{\tau}\right)  \left[
\left(  z-a\right)  ^{p-1}\left(  t-b\right)  ^{q-1}E_{\gamma_{4},\gamma_{5},\gamma_{6};p;q;\upsilon}^{\left(
\gamma_{1};\gamma_{2};\gamma_{3}\right)
}\left(  w_{1}\left(  z-a\right)  ,w_{2}\left(  t-b\right)  \right)  \right]
\label{modEdoublefracder}\\
&  =\left(  z-a\right)  ^{p-\sigma-1}\left(  t-b\right)  ^{q-\tau
-1}E_{\gamma_{4},\gamma_{5},\gamma_{6};p-\sigma;q-\tau;\upsilon}^{\left(  \gamma_{1};\gamma_{2};\gamma
_{3}\right)  }\left(  w_{1}\left(
z-a\right)  ,w_{2}\left(  t-b\right)  \right)  .\nonumber
\end{align}

\end{theorem}

\begin{proof} For $\ \operatorname{Re}%
\left(  \sigma\right)  \geq0$, $\ \operatorname{Re}%
\left(  p-\sigma\right)  \geq0$ and $\operatorname{Re}\left(  \tau\right)  \geq0$, $\operatorname{Re}\left(  q-\tau\right)  \geq0$,
\begin{align*}
&  \left(  \ _{z}D_{a^{+}}^{\sigma}\ _{t}D_{b^{+}}^{\tau}\right)  \left[
\left(  z-a\right)  ^{p-1}\left(  t-b\right)  ^{q-1}E_{\gamma_{4},\gamma_{5},\gamma_{6};p;q;\upsilon}^{\left(
\gamma_{1};\gamma_{2};\gamma_{3}\right)
}\left(  w_{1}\left(  z-a\right)  ,w_{2}\left(  t-b\right)  \right)  \right]
\\
&  =\frac{1}{\Gamma\left(  k-\sigma\right)  }\frac{1}{\Gamma\left(
s-\tau\right)  }\sum\limits_{l=0}^{\infty}\sum\limits_{n=0}^{\infty}%
\frac{\left(  -1\right)  ^{l}\left(  \gamma_{1}\right)  _{2l}\left(
\gamma_{2}\right)  _{l+n}\left(  \gamma_{3}\right)  _{l}w_{1}^{l}%
w_{2}^{\upsilon n}}{\left(  \gamma_{4}\right)  _{l}\left(  \gamma_{5}\right)
_{l}\left(  \gamma_{6}\right)  _{l}\Gamma\left(  p+l\right)  \Gamma\left(
q+\upsilon n\right)  l!n!}\\
&  \times D_{z}^{k}\int\limits_{a}^{z}\left(  z-\xi\right)  ^{k-\sigma
-1}\left(  \xi-a\right)  ^{p+l-1}d\xi\ D_{t}^{s}\int\limits_{b}^{t}\left(
t-y\right)  ^{s-\tau-1}\left(  y-b\right)  ^{q+\upsilon n-1}dy\\
&  =\left(  z-a\right)  ^{p-\sigma-1}\left(  t-b\right)  ^{q-\tau-1}%
\sum\limits_{l=0}^{\infty}\sum\limits_{n=0}^{\infty}\frac{\left(  \gamma
_{1}\right)  _{2l}\left(  \gamma_{2}\right)  _{l+n}\left(  \gamma_{3}\right)
_{l}\left(  -w_{1}\left(  z-a\right)  \right)  ^{l}\left(  w_{2}\left(
t-b\right)  \right)  ^{\upsilon n}}{\left(  \gamma_{4}\right)  _{l}\left(
\gamma_{5}\right)  _{l}\left(  \gamma_{6}\right)  _{l}\Gamma\left(
p-\sigma+l\right)  \Gamma\left(  q-\tau+\upsilon n\right)  l!n!}\\
&  =\left(  z-a\right)  ^{p-\sigma-1}\left(  t-b\right)  ^{q-\tau
-1}E_{\gamma_{4},\gamma_{5},\gamma_{6};p-\sigma;q-\tau;\upsilon}^{\left(  \gamma_{1};\gamma_{2};\gamma
_{3}\right)  }\left(  w_{1}\left(
z-a\right)  ,w_{2}\left(  t-b\right)  \right)  .
\end{align*}

\end{proof}

\begin{corollary}
For $\operatorname{Re}\left(  \sigma\right)  \geq0$ and $\operatorname{Re}%
\left(  \tau\right)  \geq0$, the following double partial fractional
derivative representation holds true for the I - Konhauser polynomials:%
\begin{align*}
&  \left(  \ _{z}D_{a^{+}}^{\sigma}\ _{t}D_{b^{+}}^{\tau}\right)  \left[
\left(  t-b\right)  ^{q}\left(  z-a\right)  ^{p}\left(  \frac{-i}{\sqrt
{w_{1}\left(  z-a\right)  }}\right)  ^{-k}\ _{K}I_{k;\upsilon}^{\left(
p,q;\gamma;c\right)  }\left(  \frac{i}{2\sqrt{w_{1}\left(  z-a\right)  }%
},w_{2}\left(  t-b\right)  \right)  \right] \\
&  =\left(  z-a\right)  ^{p-\sigma}\left(  t-b\right)  ^{q-\tau}\left(
\frac{-i}{\sqrt{w_{1}\left(  z-a\right)  }}\right)  ^{-k}\ _{K}I_{k;\upsilon
}^{\left(  p-\sigma,q-\tau;\gamma;c\right)  }\left(  \frac{i}{2\sqrt
{w_{1}\left(  z-a\right)  }},w_{2}\left(  t-b\right)  \right)  .
\end{align*}

\end{corollary}

\begin{proof}
To prove the desired, the relationship (\ref{modIE}) is used in (\ref{modEdoublefracder}).
\end{proof}
\bigskip
\subsection{Convolution type integral equation and integral operator}

We know that the double fractional integral $\left(  \ _{z}\mathbb{I}_{0^{+}%
}^{\sigma}\ _{t}\mathbb{I}_{0^{+}}^{\tau}d\right)  \left(  z,t\right)  $ may
be expressed of convolution form%
\[
\left(  \ _{z}\mathbb{I}_{0^{+}}^{\sigma}\ _{t}\mathbb{I}_{0^{+}}^{\tau
}d\right)  \left(  z,t\right)  =\left[  d\left(  z,t\right)  \ast\frac{z_{\xi
}^{\sigma-1}t_{y}^{\tau-1}}{\Gamma\left(  \sigma\right)  \Gamma\left(
\tau\right)  }\right]
\]
for $\operatorname{Re}\left(  \sigma\right)  >0$ and $\operatorname{Re}\left(
\tau\right)  >0$. Thus, we have%
\begin{equation}
\mathbb{L}_{2}\left(  \ _{z}\mathbb{I}_{0^{+}}^{\sigma}\ _{t}\mathbb{I}%
_{0^{+}}^{\tau}d\right)  \left(  z,t\right)  =z^{-\sigma}t^{-\tau}%
\mathbb{L}_{2}\left(  d\right)  \left(  z,t\right)  . \label{LI}%
\end{equation}

Let us consider the double convolution equation%
\begin{align}
\psi\left(  z,t\right)   &  =w_{1}^{k/2}\int\limits_{0}^{z}\int\limits_{0}%
^{t}\left(  z-\xi\right)  ^{p+\frac{k}{2}}\left(  t-y\right)  ^{q}d\left(
\xi,y\right) \ _{K}I_{k;\upsilon}^{\left(  p,q;\frac{-k+1}{2};p-k\right)  }\left(
\frac{1}{2\sqrt{w_{1}\left(  z-\xi\right)  }},w_{2}\left(  t-y\right)
\right)  dyd\xi. \label{IE}
\end{align}
We can give the following theorem for the solution of the integral equation
(\ref{IE}):

\begin{theorem}
The singular double integral (\ref{IE}) has the solution%
\begin{align}
d\left(  z,t\right)   &  =\frac{\left(  -1\right)  ^{k}}{\left(  1-p\right)
_{k}}\int\limits_{0}^{z}\int\limits_{0}^{t}\left(  z-\xi\right)  ^{\sigma
-p-2}\left(  t-y\right)  ^{\tau-q-2}E_{\frac{k+1}{2},k,\gamma_{3};\sigma-p-1;\tau-q-1;\upsilon}^{\left(
k;k;\gamma_{3}\right)  }\left(  w_{1}\left(
z-\xi\right)  ,w_{2}\left(  t-y\right)  \right) \label{DIE}\\
&  \times\ _{z}\mathbb{I}_{0^{+}}^{-\sigma}\ _{t}\mathbb{I}_{0^{+}}^{-\tau
}\psi\left(  \xi,y\right)  dyd\xi\nonumber
\end{align}
if$\ _{z}\mathbb{I}_{0^{+}}^{-\sigma}\ _{t}\mathbb{I}_{0^{+}}^{-\tau}\psi$
exists for $\operatorname{Re}\left(  \sigma-p-1\right)  >0  $ and $\operatorname{Re}\left(  \tau-q-1\right) >0 $ and is integrable for $0<z,t<\infty$.
\end{theorem}

\begin{proof}
Applying the double Laplace transform on both sides of (\ref{IE}) and using
convolution theorem and \textit{Corollary 45,} we have%
\[
\frac{\left(  1-p\right)  _{k}}{a^{p+1}b^{q+1}}\left(  \frac{w_{2}^{\upsilon}%
}{q^{\upsilon}}-1\right)  ^{k}\left(  1-\frac{4q^{\upsilon}w_{1}}{a\left(
w_{2}^{\upsilon}-q^{\upsilon}\right)  }\right)  ^{\frac{k}{2}}\mathbb{L}%
_{2}\left[  d\left(  \xi,y\right)  \right]  \left(  a,b\right)  =\mathbb{L}%
_{2}\left[  \psi\left(  \xi,y\right)  \right]  \left(  a,b\right)  ,
\]
and then%
\begin{equation}
\mathbb{L}_{2}\left[  d\left(  \xi,y\right)  \right]  =\frac{a^{p-\sigma
+1}b^{q-\tau+1}}{\left(  1-p\right)  _{k}}\left(  \frac{w_{2}^{\upsilon}%
}{q^{\upsilon}}-1\right)  ^{-k}\left(  1-\frac{4q^{\upsilon}w_{1}}{a\left(
w_{2}^{\upsilon}-q^{\upsilon}\right)  }\right)  ^{-\frac{k}{2}}a^{\sigma
}b^{\tau}\mathbb{L}_{2}\left[  \psi\left(  \xi,y\right)  \right]  \left(
a,b\right)  . \label{L2}%
\end{equation}
Taking inverse double Laplace transform of (\ref{L2}) and taking into account
(\ref{LI}) and \textit{Theorem 44}, we get (\ref{DIE}).
\end{proof}

We consider the following double fractional integral operator:%
\begin{align}
\left(  \mathfrak{E}_{\gamma_{4}, \gamma_{5},\gamma_{6};p;q;\upsilon;w_{1},w_{2};a_{1}^{+},b_{1}^{+}}^{\left(
\gamma_{1};\gamma_{2};\gamma_{3}\right)
}d\right)  \left(  z,t\right)  =\int\limits_{b_{1}}^{t}\int\limits_{a_{1}}^{z}\left(
t-y\right)  ^{q-1}\left(  z-\xi\right)  ^{p-1}%
\ \ \ \ \ \ \ \ \ \ \ \ \ \ \ \ \ \ \ \ \  & \label{IO}\\
\times E_{\gamma
_{4},\gamma_{5},\gamma_{6};p;q;\upsilon}^{\left(  \gamma_{1};\gamma_{2};\gamma_{3}\right)  }\left(  w_{1}\left(  z-\xi\right)
,w_{2}\left(  t-y\right)  \right)  d\left(  \xi,y\right)  d\xi dy.  &
\nonumber
\end{align}
For $\gamma_{4}=\frac{\gamma_{1}+1}{2},\gamma_{5}=\gamma_{2},\gamma_{6}=\gamma_{3}$, taking $\gamma_{1}=\gamma_{2}=\gamma_{3}=0$, the
integral operator $\mathfrak{E}_{\gamma_{4},\gamma_{5},\gamma
_{6};p;q;\upsilon;w_{1},w_{2};a_{1}^{+},b_{1}^{+}%
}^{\left(  \gamma_{1};\gamma_{2};\gamma_{3}\right)  }$ reduces to the Riemann-Liouville double fractional integral
operator defined in%
\[
\ _{z}\mathbb{I}_{a_{1}^{+}}^{\sigma}\ _{t}\mathbb{I}_{b_{1}^{+}}^{\tau
}d\left(  z,t\right)  =\frac{1}{\Gamma\left(  \sigma\right)  \Gamma\left(
\tau\right)  }\int\limits_{b_{1}}^{t}\int\limits_{a_{1}}^{z}\left(
z-\xi\right)  ^{\sigma-1}\left(  t-y\right)  ^{\tau-1}d\left(  \xi,y\right)
d\xi dy,
\]
where $z>a_{1},t>b_{1}$ and$\ \operatorname{Re}\left(  \sigma\right)
>0,\operatorname{Re}\left(  \tau\right)  >0$, i.e.%
\begin{equation}
\left(  \mathfrak{E}_{\frac{1}{2},0,0;p;q;\upsilon;w_{1},w_{2};a_{1}^{+},b_{1}^{+}}^{\left(
0;0;0\right)  }d\right)  \left(  z,t\right)  =\left(  \ _{t}%
\mathbb{I}_{b_{1}^{+}}^{q}\ _{z}\mathbb{I}_{a_{1}^{+}}^{p}d\right)  \left(
z,t\right)  . \label{op0}%
\end{equation}
Now, we give the transformation properties of $\mathfrak{E}_{\gamma_{4},\gamma_{5},\gamma_{6};p;q;\upsilon
;w_{1},w_{2};a_{1}^{+},b_{1}^{+}}^{\left(  \gamma_{1};\gamma_{2};\gamma
_{3}\right)  }$ in the space $L\left(
\left(  a_{1},a_{2}\right)  \times\left(  b_{1},b_{2}\right)  \right)  $ of
Lebesgue measurable functions where, as usual%
\[
L\left(  \left(  a_{1},a_{2}\right)  \times\left(  b_{1},b_{2}\right)
\right)  =\left\{  d:\left\Vert d\right\Vert _{1}:=\int\limits_{a_{1}}^{a_{2}%
}\int\limits_{b_{1}}^{b_{2}}\left\vert d\left(  z,t\right)  \right\vert
dtdz<\infty\right\}  .
\]

\begin{theorem}
The double integral operator $\mathfrak{E}_{\gamma_{4},\gamma_{5},\gamma_{6};p;q;\upsilon;w_{1},w_{2};a_{1}%
^{+},b_{1}^{+}}^{\left(  \gamma_{1};\gamma_{2};\gamma_{3}\right)  }$ is bounded in the space $L\left(  \left(
a_{1},a_{2}\right)  \times\left(  b_{1},b_{2}\right)  \right)  $, i.e.%
\[
\left\Vert \mathfrak{E}_{\gamma_{4},\gamma_{5},\gamma_{6};p;q;\upsilon;w_{1},w_{2};a_{1}^{+},b_{1}^{+}%
}^{\left(  \gamma_{1};\gamma_{2};\gamma_{3}\right)  }d\right\Vert _{1}\leq\mathcal{M}\left\Vert d\right\Vert _{1},
\]
where the constant $\mathcal{M}$ $\left(  0<\mathcal{M}<\infty\right)  $ is as
follows%
\begin{align}
\mathcal{M}=\left(  a_{2}-b_{1}\right)  ^{\operatorname{Re}\left(  p\right)
}\left(  b_{2}-a_{1}\right)  ^{\operatorname{Re}\left(  q\right)  }%
\sum\limits_{n=0}^{\infty}\sum\limits_{l=0}^{\infty}\frac{\left\vert \left(
\gamma_{1}\right)  _{2l}\right\vert \left\vert \left(  \gamma_{2}\right)
_{l+n}\right\vert \left\vert \left(  \gamma_{3}\right)  _{l}\right\vert
}{\left\vert \left(  \gamma_{4}\right)  _{l}\right\vert \left\vert \left(
\gamma_{5}\right)  _{l}\right\vert \left\vert \left(  \gamma_{6}\right)
_{l}\right\vert }  & \label{con}\\
\times\frac{\left\vert w_{1}\left(  a_{2}-b_{1}\right)  \right\vert
^{l}\left\vert w_{2}\left(  b_{2}-a_{1}\right)  \right\vert ^{\upsilon n}%
}{\left\vert \Gamma\left(  p+l\right)  \right\vert \left\vert \Gamma\left(
q+\upsilon n\right)  \right\vert \left(  \operatorname{Re}\left(  p\right)
+l\right)  \left(  \operatorname{Re}\left(  q\right)  +\upsilon n\right)
l!n!}<\infty.  & \nonumber
\end{align}

\end{theorem}

\begin{proof}
By using the Fubini's theorem, we get%
\begin{align*}
\left\Vert \mathfrak{E}_{\gamma_{4},\gamma_{5},\gamma_{6};p;q;\upsilon;w_{1},w_{2};a_{1}^{+},b_{1}^{+}%
}^{\left(  \gamma_{1};\gamma_{2};\gamma_{3}\right)  }d\right\Vert _{1}\leq\int\limits_{a_{1}}^{a_{2}}\int%
\limits_{b_{1}}^{b_{2}}\left\vert d\left(  \xi,y\right)  \right\vert
\int\limits_{\xi}^{a_{2}}\int\limits_{y}^{b_{2}}\left(  z-\xi\right)
^{\operatorname{Re}\left(  p\right)  -1}\left(  t-y\right)
^{\operatorname{Re}\left(  q\right)  -1}%
\ \ \ \ \ \ \ \ \ \ \ \ \ \ \ \ \ \ \ \ \ \ \ \ \ \  & \\
\times\left\vert E_{\gamma_{4},\gamma_{5},\gamma_{6};p;q;\upsilon}^{\left(  \gamma_{1};\gamma_{2};\gamma
_{3}\right)  }\left(  w_{1}\left(
z-\xi\right)  ,w_{2}\left(  t-y\right)  \right)  \right\vert dtdzdyd\xi & \\
=\int\limits_{b_{1}}^{b_{2}}\int\limits_{a_{1}}^{a_{2}}\left\vert d\left(
\xi,y\right)  \right\vert \int\limits_{0}^{b_{2}-y}\int\limits_{0}^{a_{2}-\xi
}u^{\operatorname{Re}\left(  p\right)  -1}x^{\operatorname{Re}\left(
q\right)  -1}\left\vert E_{\gamma_{4},\gamma_{5},\gamma_{6};p;q;\upsilon}^{\left(  \gamma_{1};\gamma_{2}%
;\gamma_{3}\right)  }\left(  w_{1}%
u,w_{2}x\right)  \right\vert dudxd\xi dy\ \ \ \ \  & \\
\leq\int\limits_{b_{1}}^{b_{2}}\int\limits_{a_{1}}^{a_{2}}\left\vert d\left(
\xi,y\right)  \right\vert \int\limits_{0}^{b_{2}-a_{1}}\int\limits_{0}%
^{a_{2}-b_{1}}u^{\operatorname{Re}\left(  p\right)  -1}x^{\operatorname{Re}%
\left(  q\right)  -1}\left\vert E_{\gamma_{4},\gamma_{5},\gamma_{6};p;q;\upsilon}^{\left(  \gamma_{1}%
;\gamma_{2};\gamma_{3}\right)  }\left(
w_{1}u,w_{2}x\right)  \right\vert dudxd\xi dy  & \\
=\int\limits_{a_{1}}^{a_{2}}\int\limits_{b_{1}}^{b_{2}}\left\vert d\left(
\xi,y\right)  \right\vert dyd\xi\sum\limits_{l=0}^{\infty}\sum\limits_{n=0}%
^{\infty}\frac{\left\vert \left(  \gamma_{1}\right)  _{2l}\right\vert
\left\vert \left(  \gamma_{2}\right)  _{l+n}\right\vert \left\vert \left(
\gamma_{3}\right)  _{l}\right\vert w_{1}^{l}w_{2}^{\upsilon n}}{\left\vert
\left(  \gamma_{4}\right)  _{l}\right\vert \left\vert \left(  \gamma
_{5}\right)  _{l}\right\vert \left\vert \left(  \gamma_{6}\right)
_{l}\right\vert \left\vert \Gamma\left(  p+l\right)  \right\vert \left\vert
\Gamma\left(  q+\upsilon n\right)  \right\vert n!l!}%
\ \ \ \ \ \ \ \ \ \ \ \ \ \ \ \  & \\
\times\int\limits_{0}^{a_{2}-b_{1}}\int\limits_{0}^{b_{2}-a_{1}}%
u^{\operatorname{Re}\left(  p\right)  +l-1}x^{\operatorname{Re}\left(
q\right)  +\upsilon n-1}%
dxdu\ \ \ \ \ \ \ \ \ \ \ \ \ \ \ \ \ \ \ \ \ \ \ \ \ \ \ \ \ \ \ \ \ \ \ \ \ \ \ \ \ \ \ \ \ \ \ \ \ \ \ \
& \\
=\mathcal{M}\left\Vert d\right\Vert _{1}%
.\ \ \ \ \ \ \ \ \ \ \ \ \ \ \ \ \ \ \ \ \ \ \ \ \ \ \ \ \ \ \ \ \ \ \ \ \ \ \ \ \ \ \ \ \ \ \ \ \ \ \ \ \ \ \ \ \ \ \ \ \ \ \ \ \ \ \ \ \ \ \ \ \ \ \ \ \ \ \ \ \ \ \ \ \ \ \ \ \ \ \ \ \ \ \ \ \ \ \ \
&
\end{align*}

\end{proof}

\begin{remark}
The constant $M$ is finite since the series%
\[
\sum\limits_{l=0}^{\infty}\sum\limits_{n=0}^{\infty}\frac{\left(  \gamma
_{1}\right)  _{2l}\left(  \gamma_{2}\right)  _{l+n}\left(  \gamma_{3}\right)
_{l}\ z^{l}t^{\upsilon n}}{\left(  \gamma_{4}\right)  _{l}\left(  \gamma
_{5}\right)  _{l}\left(  \gamma_{6}\right)  _{l}\Gamma\left(  p+l\right)
\Gamma\left(  q+\upsilon n\right)  n!l!}%
\]
is absolutely convergent for all $z$ and $t$ (see \cite{SD}).
\end{remark}

Now we show that the integral operator $\mathfrak{E}_{\gamma_{4},\gamma_{5},\gamma_{6};p;q;\upsilon;w_{1}%
,w_{2};a_{1}^{+},b_{1}^{+}}^{\left(  \gamma_{1};\gamma_{2};\gamma_{3}%
\right)  }$ is bounded in the space $C\left(
\left[  a_{1},a_{2}\right]  \times\left[  b_{1},b_{2}\right]  \right)  $ of
continuous functions on $\left[  a_{1},a_{2}\right]  \times\left[  b_{1}%
,b_{2}\right]  $ with a max norm%
\begin{equation}
\left\Vert h\right\Vert _{C}=\max_{\substack{a_{1}\leq z\leq a_{2}\\b_{1}\leq
t\leq b_{2}}}\left\vert h\left(  z,t\right)  \right\vert . \label{hmax}%
\end{equation}

\begin{theorem}
The double integral operator $\mathfrak{E}_{\gamma_{4}
\gamma_{5}
\gamma_{6};p;q;\upsilon;w_{1},w_{2};a_{1}%
^{+},b_{1}^{+}}^{\left(  \gamma_{1};\gamma_{2};\gamma_{3}\right)  }$ is bounded in the space $C\left(  \left[
a_{1},a_{2}\right]  \times\left[  b_{1},b_{2}\right]  \right)  $, i.e.%
\begin{equation}
\left\Vert \mathfrak{E}_{\gamma_{4},\gamma_{5},\gamma_{6};p;q;\upsilon;w_{1},w_{2};a_{1}^{+},b_{1}^{+}%
}^{\left(  \gamma_{1};\gamma_{2};\gamma_{3}\right)  }d\right\Vert _{C}\leq\mathcal{M}\left\Vert d\right\Vert _{C},
\label{BIOC}%
\end{equation}
where $\mathcal{M}$ is given by (\ref{con}).
\end{theorem}

\begin{proof}
Considering (\ref{IO}) and (\ref{hmax}), we can obtain (\ref{BIOC}).
\end{proof}

\bigskip

In the next theorem, by using double Laplace transform of (\ref{IO}) we
construct double integrals involving the product of two bivariate
Mittag-Leffler function $E_{\gamma_{4},\gamma_{5},\gamma_{6};p;q;\upsilon}^{\left(  \gamma_{1};\gamma
_{2};\gamma_{3}\right)  }$.

\begin{theorem}
Assume that $p,q,\upsilon,\sigma,\tau,w_{1},w_{2},\gamma_{1},\gamma_{2},\gamma_{3}%
,\mu_{1},\mu_{2},\mu_{3}\in%
\mathbb{C}
$, and $\operatorname{Re}\left(p \right)>0,\operatorname{Re}\left(q \right)>0,\operatorname{Re}\left(\upsilon \right)>0,\operatorname{Re}\left(\sigma \right)>0,\operatorname{Re}\left(\tau \right)>0$. Then we get%
\begin{align}
&  \int\limits_{0}^{t}\int\limits_{0}^{z}\left(  z-\xi\right)  ^{p-1}\left(
t-y\right)  ^{q-1}E_{\frac{\gamma_{1}+1}{2},\gamma_{2},\gamma_{3};p;q;\upsilon}^{\left(  \gamma_{1};\gamma_{2};\gamma
_{3}\right)  }\left(  w_{1}\left(
z-\xi\right)  ,w_{2}\left(  t-y\right)  \right) \label{DIT}\\
&  \times\xi^{\sigma-1}y^{\tau-1}E_{\frac{\mu_{1}+1}{2},\mu_{2},\mu_{3};\sigma;\tau;\upsilon}^{\left(  \mu_{1}%
;\mu_{2};\mu_{3}\right)  }\left(  w_{1}%
\xi,w_{2}y\right)  d\xi dy\nonumber\\
&  =z^{p+\sigma-1}t^{q+\tau-1}E_{\frac{\gamma_{1}+\mu_{1}+1}{2},\gamma_{2}+\mu_{2},\gamma_{3}+\mu_{3};p+\sigma;q+\tau;\upsilon}^{\left(  \gamma
_{1}+\mu_{1};\gamma_{2}+\mu_{2};\gamma_{3}+\mu_{3}\right)  }\left(  w_{1}%
z,w_{2}t\right)  .\nonumber
\end{align}

\end{theorem}

\begin{proof}
With the help of the double convolution theorem for the double Laplace
transform in the left hand side of (\ref{DIT}), from \textit{Theorem 44,} we
derive%
\begin{align}
&  \mathbb{L}_{2}\left\{  \int\limits_{0}^{t}\int\limits_{0}^{z}\left(
z-\xi\right)  ^{p-1}\left(  t-y\right)  ^{q-1}E_{\frac{\gamma_{1}+1}{2},\gamma_{2},\gamma
_{3};p;q;\upsilon}^{\left(
\gamma_{1};\gamma_{2};\gamma_{3}\right)  }\left(  w_{1}\left(  z-\xi\right)  ,w_{2}\left(  t-y\right)
\right)  \right. \label{op}\\
&  \left.  \times\xi^{\sigma-1}y^{\tau-1}E_{\frac{\mu_{1}+1}{2},\mu_{2},\mu_{3};\sigma;\tau;\upsilon}^{\left(
\mu_{1};\mu_{2};\mu_{3}\right)  }\left(
w_{1}\xi,w_{2}y\right)  d\xi dy\right\} \nonumber\\
&  =\frac{1}{a^{p+\sigma}b^{q+\tau}}\left(  1-\frac{w_{2}^{\upsilon}%
}{b^{\upsilon}}\right)  ^{-\gamma_{2}-\mu_{2}}\left(  1-\frac{4w_{1}%
b^{\upsilon}}{a\left(  w_{2}^{\upsilon}-b^{\upsilon}\right)  }\right)
^{-\frac{\gamma_{1}}{2}-\frac{\mu_{1}}{2}}\nonumber\\
&  =\mathbb{L}_{2}\left\{  z^{p+\sigma-1}t^{q+\tau-1}E_{\frac{\gamma_{1}+\mu_{1}+1}{2},\gamma_{2}+\mu_{2},\gamma_{3}+\mu_{3};p+\sigma
;q+\tau;\upsilon}^{\left(  \gamma_{1}+\mu_{1};\gamma_{2}+\mu_{2};\gamma
_{3}+\mu_{3}\right)  }\left(  w_{1}z,w_{2}t\right)  \right\}  ,\nonumber
\end{align}
for $\left\vert \frac{w_{2}^{\upsilon}}{b^{\upsilon}}\right\vert <1$,
$\left\vert \frac{4w_{1}b^{\upsilon}}{a\left(  w_{2}^{\upsilon}-b^{\upsilon
}\right)  }\right\vert <1$, $\operatorname{Re}\left(  a\right)  >0$ and
$\operatorname{Re}\left(  b\right)  >0$.

Taking inverse Laplace transform on both sides of (\ref{op}), we write the
equation in (\ref{DIT}).
\end{proof}

\bigskip

Now, we give the following theorem about the composition of two operators with
different indices.

\begin{remark}
For the operations in this section, we should note that the modified I - Konhauser Mittag-Leffler functions are reduced to a simpler form by cancelling the parameters taken equal to each other. That is, from (\ref{ModEdef}), we obtain
\begin{align*}
E_{\frac{\gamma_{1}+1}{2}%
,\gamma_{2},\gamma_{3};p;q;\upsilon}^{\left(  \gamma_{1};\gamma_{2};\gamma_{3}\right)  }\left(  z,t\right)  =\sum_{l=0}^{\infty}%
\sum_{n=0}^{\infty}\frac{\left(  -1\right)  ^{l}\left(  \gamma_{1}\right)
_{2l}\left(  \gamma_{2}\right)  _{l+n}\left(  \gamma_{3}\right)  _{l}%
z^{l}t^{\upsilon n}}{\left( \frac{\gamma_{1}+1}{2}\right)  _{l}\left(  \gamma_{2}\right)
_{l}\left(  \gamma_{3}\right)  _{l}\Gamma\left(  p+l\right)  \Gamma\left(
q+\upsilon n\right)  n!l!}\\
=\sum_{l=0}^{\infty}%
\sum_{n=0}^{\infty}\frac{\left(  -1\right)  ^{l} 2^{2l}\left( \frac{\gamma_{1}}{2}\right)  _{l}\left(  \gamma_{2}+l\right)  _{n}%
z^{l}t^{\upsilon n}}{\Gamma\left(  p+l\right)  \Gamma\left(
q+\upsilon n\right)  n!l!}.
\end{align*}
After this simplification, the parameters can be replaced by constants.\\
It should be noted that this order will be taken into account for all operations in the rest of the work.
\end{remark}

\begin{theorem}
Assume that $p,q,\upsilon,\sigma,\tau,w_{1},w_{2},\gamma_{1},\gamma_{2},\gamma_{3}
,\mu_{1},\mu_{2},\mu_{3}\in%
\mathbb{C}
$, and $\operatorname{Re}\left(p \right)>0,\operatorname{Re}\left(q \right)>0,\operatorname{Re}\left(\upsilon \right)>0,\operatorname{Re}\left(\sigma \right)>0,\operatorname{Re}\left(\tau \right)>0$, then we have the relation%
\begin{align}
\left(  \mathfrak{E}_{\frac{\gamma_{1}+1}{2},\gamma_{2},\gamma_{3};p,q,\upsilon;w_{1},w_{2};a_{1}^{+},b_{1}^{+}}^{\left(
\gamma\right) }\mathfrak{E}_{\frac{\mu_{1}+1}{2},\mu_{2}, \mu_{3};\sigma,\tau,\upsilon;w_{1},w_{2};a_{1}^{+}%
,b_{1}^{+}}^{\left(\mu\right)}d\right)\left(z,t\right) \label{E1} \\
=\left( \mathfrak{E}_{\frac{\gamma_{1}+\mu_{1}+1}{2},\gamma_{2}+\mu_{2},
\gamma_{3}+\mu_{3};p+\sigma,q+\tau,\upsilon;w_{1},w_{2};a_{1}^{+},b_{1}^{+}}^{\left(  \gamma+\mu\right)  }d\right)
\left(  z,t\right) \nonumber
\end{align}
for any summable function $g\in L\left(  \left(  a_{1},a_{2}\right)
\times\left(  b_{1},b_{2}\right)  \right)  $, where
$\gamma:=\left(  \gamma_{1},\gamma_{2},\gamma_{3}\right)$
and $\mu:=\left(  \mu_{1},\mu_{2},\mu_{3}\right).$ \\

In particular,%
\begin{equation}
\left(  \mathfrak{E}_{\frac{\gamma_{1}+1}{2}%
;\gamma_{2};\gamma_{3};p,q,\upsilon;w_{1},w_{2};a_{1}^{+},b_{1}^{+}}^{\left(
\gamma\right)  }\mathfrak{E}_{\frac{-\gamma_{1}+1}{2}%
;-\gamma_{2};-\gamma_{3};\sigma,\tau,\upsilon;w_{1},w_{2};a_{1}^{+}%
,b_{1}^{+}}^{\left(  -\gamma\right)  }d\right)  \left(  z,t\right)  =\left(
\ _{t}\mathbb{I}_{b_{1}^{+}}^{q+\tau}\ _{z}\mathbb{I}_{a_{1}^{+}}^{p+\sigma
}d\right)  \left(  z,t\right)  \mathfrak{.} \label{E2}%
\end{equation}
\end{theorem}

\begin{proof}
Considering (\ref{IO}), we get
\begin{align*}
&  \left(  \mathfrak{E}_{\frac{\gamma_{1}+1}{2};\gamma_{2};\gamma
_{3};p,q,\upsilon;w_{1},w_{2};a_{1}^{+},b_{1}^{+}}^{\left(
\gamma\right)  }\mathfrak{E}_{\frac{\mu_{1}+1}{2};\mu_{2};\mu
_{3};\sigma,\tau,\upsilon;w_{1},w_{2};a_{1}^{+}%
,b_{1}^{+}}^{\left(  \mu\right)  }d\right)  \left(  z,t\right)\\
&  =\left(
\mathfrak{E}_{\frac{\gamma_{1}+1}{2};\gamma_{2};\gamma
_{3};p,q,\upsilon;w_{1},w_{2};a_{1}^{+},b_{1}^{+}}^{\left(
\gamma_{1};\gamma_{2};\gamma_{3}\right)  }\mathfrak{E}_{\frac{\mu_{1}+1}{2};\mu_{2};\mu
_{3};\sigma,\tau,\upsilon;w_{1},w_{2};a_{1}^{+}%
,b_{1}^{+}}^{\left(  \mu_{1};\mu_{2};\mu_{3}\right)  }d\right)  \left(  z,t\right) \\
&  =\int\limits_{b_{1}}^{t}\int\limits_{a_{1}}^{z}\left(  z-\xi\right)
^{p-1}\left(  t-y\right)  ^{q-1}E_{\frac{\gamma_{1}+1}{2};\gamma_{2};\gamma_{3};p,q,\upsilon}^{\left(  \gamma_{1}%
;\gamma_{2};\gamma_{3}\right)
}\left(  w_{1}\left(  z-\xi\right)  ,w_{2}\left(  t-y\right)  \right) \mathfrak{E}_{\frac{\mu_{1}+1}{2};\mu_{2};\mu
_{3};\sigma,\tau,\upsilon;w_{1},w_{2};a_{1}^{+},b_{1}^{+}%
}^{\left(  \mu_{1};\mu_{2};\mu_{3}\right)
}d\left(  \xi,y\right)  d\xi dy\\
& =\int\limits_{b_{1}}^{t}\int\limits_{a_{1}}^{z}\int\limits_{b_{1}}^{y}%
\int\limits_{a_{1}}^{\xi}\left(  z-\xi\right)  ^{p-1}\left(  t-y\right)
^{q-1}E_{\frac
{\gamma_{1}+1}{2};\gamma_{2};\gamma_{3};p,q,\upsilon}^{\left(  \gamma_{1};\gamma_{2};\gamma_{3}\right)  }\left(  w_{1}\left(
z-\xi\right)  ,w_{2}\left(  t-y\right)  \right)  \left(  \xi-u\right)
^{\sigma-1}  \\
& \ \ \ \ \  \times\left(  y-x\right)  ^{\tau-1}E_{\frac{\mu_{1}+1}{2};\mu_{2};\mu_{3};\sigma,\tau,\upsilon}^{\left(  \mu
_{1};\mu_{2};\mu_{3}\right)  }\left(
w_{1}\left(  \xi-u\right)  ,w_{2}\left(  y-x\right)  \right)  d\left(
u,x\right)  dudxd\xi dy. 
\end{align*}
Thus,%
\begin{align}
&  \left(  \mathfrak{E}_{\frac
{\gamma_{1}+1}{2};\gamma_{2};\gamma_{3};p,q,\upsilon;w_{1},w_{2};a_{1}^{+},b_{1}^{+}%
}^{\left(  \gamma\right)  }\mathfrak{E}_{\frac{\mu_{1}+1}{2};\mu_{2};\mu_{3};\sigma,\tau,\upsilon;w_{1}%
,w_{2};a_{1}^{+},b_{1}^{+}}^{\left(  \mu\right)  }d\right)  \left(  z,t\right)
\label{opop}\\
&  =\int\limits_{b_{1}}^{t}\int\limits_{a_{1}}^{z}\int\limits_{x}^{t}%
\int\limits_{u}^{z}\left(  z-\xi\right)  ^{p-1}\left(  t-y\right)
^{q-1}E_{\frac
{\gamma_{1}+1}{2};\gamma_{2};\gamma_{3};p,q,\upsilon}^{\left(  \gamma_{1};\gamma_{2};\gamma_{3}\right)  }\left(  w_{1}\left(
z-\xi\right)  ,w_{2}\left(  t-y\right)  \right)  \left(  \xi-u\right)
^{\sigma-1}\nonumber\\
&  \times\left(  y-x\right)  ^{\tau-1}E_{\frac{\mu_{1}+1}{2};\mu_{2};\mu_{3};\sigma,\tau,\upsilon}^{\left(
\mu_{1};\mu_{2};\mu_{3}\right)  }\left(
w_{1}\left(  \xi-u\right)  ,w_{2}\left(  y-x\right)  \right)  d\left(
u,x\right)  d\xi dydudx\nonumber\\
&  =\int\limits_{a_{1}}^{z}\int\limits_{b_{1}}^{t}\int\limits_{0}^{z-u}%
\int\limits_{0}^{t-x}\left(  z-k-u\right)  ^{p-1}\left(  t-s-x\right)
^{q-1}E_{\frac
{\gamma_{1}+1}{2};\gamma_{2};\gamma_{3};p,q,\upsilon}^{\left(  \gamma_{1};\gamma_{2};\gamma_{3}\right)  }\left(  w_{1}\left(
z-k-u\right)  ,w_{2}\left(  t-s-x\right)  \right) \nonumber\\
&  \times k^{\sigma-1}s^{\tau-1}E_{\frac{\mu_{1}+1}{2};\mu_{2};\mu_{3};\sigma,\tau,\upsilon}^{\left(  \mu_{1}%
;\mu_{2};\mu_{3}\right)  }\left(
w_{1}k,w_{2}s\right)  d\left(  u,x\right)  dsdkdxdu.\nonumber
\end{align}
Via \textit{Theorem 58} and (\ref{IO}), we write from (\ref{opop})%
\begin{align*}
& \left(  \mathfrak{E}_{\frac
{\gamma_{1}+1}{2};\gamma_{2};\gamma_{3};p,q,\upsilon;w_{1},w_{2};a_{1}^{+},b_{1}^{+}}^{\left(
\gamma\right)  }\mathfrak{E}_{\frac{\mu_{1}+1}{2};\mu_{2};\mu_{3};\sigma,\tau,\upsilon;w_{1},w_{2};a_{1}^{+}%
,b_{1}^{+}}^{\left(  \mu\right)  }d\right)  \left(  z,t\right) \\
& =\int \limits_{b_{1}}^{t}\int\limits_{a_{1}}^{z}d\left(  u,x\right)  \left(
t-x\right)  ^{q+\tau-1}\left(  z-u\right)  ^{p+\sigma-1} E_{\frac{\gamma_{1}+\mu_{1}+1}{2};\gamma_{2}+\mu
_{2};\gamma_{3}+\mu_{3};p+\sigma,q+\tau,\upsilon}^{\left(  \gamma_{1}+\mu_{1};\gamma_{2}%
+\mu_{2};\gamma_{3}+\mu_{3}\right)  }\left(  w_{1}\left(  z-u\right)
,w_{2}\left(  t-x\right)  \right) dudx \\
& =\mathfrak{E}_{\frac{\gamma_{1}+\mu_{1}+1}{2};\gamma_{2}+\mu
_{2};\gamma_{3}+\mu_{3};p+\sigma,q+\tau,\upsilon;w_{1},w_{2};a_{1}^{+},b_{1}^{+}%
}^{\left(  \gamma_{1}+\mu_{1};\gamma_{2}+\mu_{2};\gamma_{3}+\mu_{3}
\right)
}d\left(  z,t\right)
.\ \ \ \ \ \ \ \ \ \ \ \ \ \ \ \ \ \ \ \ \ \ \ \ \ \ \ \ \ \ \ \ \ \ \ \ \ \ \ \ \ \ \ \ \ \ \ \ \ \ \ \ \ \ \ \ \
\end{align*}
When $\mu_{1}=-\gamma_{1},\mu_{2}=-\gamma_{2},\mu_{3}=-\gamma_{3}$, (\ref{E1})
gives (\ref{E2}) considering (\ref{op0}).
\end{proof}

\begin{corollary}
If $p,q,\upsilon,\sigma,\tau,w_{1},w_{2},\mu_{1},\mu_{2},\mu_{3}\in%
\mathbb{C}
$, $\operatorname{Re}\left(  p\right)  >0$, $\operatorname{Re}\left(
q\right)  >0$, $\operatorname{Re}\left(  \upsilon\right)  >0$,
$\operatorname{Re}\left(  \sigma\right)  >0$, $\operatorname{Re}\left(
\tau\right)  >0$, then we have the following relation on $L\left(  \left(
a_{1},a_{2}\right)  \times\left(  b_{1},b_{2}\right)  \right)  $%
\begin{align*}
\mathfrak{E}_{\frac{1}{2};0;0;p,q,\upsilon;w_{1},w_{2};a_{1}^{+},b_{1}^{+}}^{\left(
0;0;0\right)  }\ \mathfrak{E}_{\frac{\mu_{1}+1}{2}%
;\mu_{2};\mu_{3};\sigma,\tau,\upsilon;w_{1},w_{2}%
;a_{1}^{+},b_{1}^{+}}^{\left(  \mu_{1};\mu_{2};\mu_{3}\right)  }  &  =\ _{t}\mathbb{I}_{b_{1}^{+}}^{q}%
\ _{z}\mathbb{I}_{a_{1}^{+}}^{p}\mathfrak{E}_{\frac{\mu_{1}+1}%
{2};\mu_{2};\mu_{3};\sigma,\tau,\upsilon;w_{1}%
,w_{2};a_{1}^{+},b_{1}^{+}}^{\left(  \mu_{1};\mu_{2};\mu_{3}\right)  }\\
&  =\mathfrak{E}_{\frac{\mu_{1}+1}{2};\mu_{2};\mu_{3};p+\sigma,q+\tau,\upsilon;w_{1},w_{2};a_{1}^{+},b_{1}^{+}%
}^{\left(  \mu_{1};\mu_{2};\mu_{3}\right)
}.
\end{align*}

\end{corollary}

\begin{corollary}
If $w_{1},w_{2},p,q,\upsilon,\sigma,\tau,\gamma_{6},\gamma_{5},\gamma
_{4},\gamma_{3},\gamma_{2},\gamma_{1}\in%
\mathbb{C}
$, $\operatorname{Re}\left(  p\right)  >0$, $\operatorname{Re}\left(
q\right)  >0$, $\operatorname{Re}\left(  \upsilon\right)  >0$,
$\operatorname{Re}\left(  \sigma\right)  >0$, $\operatorname{Re}\left(  p-\sigma\right)  >0$, $\operatorname{Re}\left(
\tau\right)  >0$, $\operatorname{Re}\left(  q-\tau\right)  >0$ and $\gamma_{4},\gamma_{5},\gamma_{6}\notin \mathbb{Z}^- \cup \{0\}$, then we get the following composition relationships:%
\begin{align*}
\left(  \ _{t}\mathbb{I}_{b_{1}^{+}}^{\tau}\ _{z}\mathbb{I}_{a_{1}^{+}%
}^{\sigma}\mathfrak{E}_{\gamma_{4};\gamma_{5};\gamma_{6};p,q,\upsilon;w_{1},w_{2};a_{1}^{+},b_{1}^{+}}^{\left(
\gamma_{1};\gamma_{2};\gamma_{3}\right)
}d\right)  \left(  z,t\right)   &  =\left(  \mathfrak{E}_{\gamma_{4};\gamma_{5};\gamma_{6};p+\sigma
,q+\tau,\upsilon;w_{1},w_{2};a_{1}^{+},b_{1}^{+}}^{\left(  \gamma_{1}%
,\gamma_{2},\gamma_{3}\right)  }d\right)
\left(  z,t\right) \\
&  =\left(  \mathfrak{E}_{\gamma_{4};\gamma_{5};\gamma_{6};p,q,\upsilon;w_{1},w_{2};a_{1}^{+},b_{1}^{+}%
}^{\left(  \gamma_{1};\gamma_{2};\gamma_{3}\right)  }\ _{t}\mathbb{I}_{b_{1}^{+}}^{\tau}\ _{z}\mathbb{I}_{a_{1}^{+}%
}^{\sigma}\right)  \left(  z,t\right)
\end{align*}
and%
\begin{align*}
\left(  \ _{t}D_{b_{1}^{+}}^{\tau}\ _{z}D_{a_{1}^{+}}^{\sigma}\mathfrak{E}%
_{\gamma_{4};\gamma_{5};\gamma_{6};p,q,\upsilon;w_{1},w_{2};a_{1}^{+},b_{1}^{+}}^{\left(  \gamma_{1};\gamma
_{2};\gamma_{3}\right)  }d\right)  \left(
z,t\right)   &  =\left(  \mathfrak{E}_{\gamma_{4};\gamma_{5};\gamma_{6};p-\sigma,q-\tau,\upsilon;w_{1}%
,w_{2};a_{1}^{+},b_{1}^{+}}^{\left(  \gamma_{1},\gamma_{2},\gamma_{3}%
\right)  }d\right)  \left(  z,t\right) \\
&  =\left(  \mathfrak{E}_{\gamma_{4};\gamma_{5};\gamma_{6};p,q,\upsilon;w_{1},w_{2};a_{1}^{+},b_{1}^{+}%
}^{\left(  \gamma_{1};\gamma_{2};\gamma_{3}\right)  }\ _{t}D_{b_{1}^{+}}^{\tau}\ _{z}D_{a_{1}^{+}}^{\sigma}d\right)
\left(  z,t\right)
\end{align*}
hold true for all $d\left(  z,t\right)  $ in $L\left(  \left(  a_{1}%
,a_{2}\right)  \times\left(  b_{1},b_{2}\right)  \right)  $.
\end{corollary}

\begin{theorem}
The integral operator $\left(  \mathfrak{E}_{\frac
{\gamma_{1}+1}{2};\gamma_{2};\gamma_{3};p,q,\upsilon,w_{1},w_{2}%
;a_{1}^{+},b_{1}^{+}}^{\left(  \gamma_{1};\gamma_{2};\gamma_{3}\right)  }d\right)  \left(  z,t\right)  $
has a left inverse operator denoted by $\left(  \mathbb{D}_{\frac
{\gamma_{1}+1}{2};\gamma_{2};\gamma_{3};p,q,\upsilon
,w_{1},w_{2};a_{1}^{+},b_{1}^{+}}^{\left(  \gamma_{1};\gamma_{2};\gamma
_{3}\right)  }d\right)  \left(
z,t\right)  $ on the space $L\left(  \left(  a_{1},a_{2}\right)  \times\left(
b_{1},b_{2}\right)  \right)  $,%
\[
\left(  \mathbb{D}_{\frac
{\gamma_{1}+1}{2};\gamma_{2};\gamma_{3};p,q,\upsilon,w_{1},w_{2};a_{1}^{+},b_{1}^{+}}^{\left(
\gamma_{1};\gamma_{2};\gamma_{3}\right)  }d\right)  \left(  z,t\right)  =\ _{t}D_{b_{1}^{+}}^{q+\tau
}\ _{z}D_{a_{1}^{+}}^{p+\sigma}\left(  \mathfrak{E}_{\frac{-\gamma_{1}+1}{2};-\gamma_{2};-\gamma_{3};p,q,\upsilon,w_{1}%
,w_{2};a_{1}^{+},b_{1}^{+}}^{\left(  -\gamma_{1};-\gamma_{2};-\gamma
_{3}\right)  }d\right)  \left(
z,t\right)  .
\]
\end{theorem}

\begin{proof}
For $\phi\in L\left(  \left(  a_{1},a_{2}\right)  \times\left(  b_{1}%
,b_{2}\right)  \right)  $,%
\begin{equation}
\phi\left(  z,t\right)  =\left(  \mathfrak{E}_{\frac{\gamma_{1}+1}{2};\gamma_{2};\gamma_{3};\sigma,\tau,\upsilon
,w_{1},w_{2};a_{1}^{+},b_{1}^{+}}^{\left(  \gamma_{1};\gamma_{2};\gamma
_{3}\right)  }d\right)  \left(
z,t\right)  . \label{phi}%
\end{equation}
Applying the operator $\left(  \mathfrak{E}_{\frac{-\gamma_{1}+1}{2};-\gamma_{2};-\gamma_{3};p,q,\upsilon,w_{1},w_{2}%
;a_{1}^{+},b_{1}^{+}}^{\left(  -\gamma_{1};-\gamma_{2};-\gamma_{3}%
\right)  }\right)  $ to
(\ref{phi}) and from the semigroup property, we obtain%
\begin{align}
\left(  \mathfrak{E}_{\frac{-\gamma_{1}+1}{2};-\gamma_{2};-\gamma_{3};p,q,\upsilon,w_{1},w_{2};a_{1}^{+},b_{1}^{+}}^{\left(
-\gamma_{1};-\gamma_{2};-\gamma_{3}\right)  }\phi\right)  \left(  z,t\right)   &  =\left(
\mathfrak{E}_{\frac{1}{2};0;0;p+\sigma,q+\tau,\upsilon,w_{1},w_{2};a_{1}^{+},b_{1}^{+}%
}^{\left(  0;0;0\right)  }d\right)  \left(  z,t\right) \label{eq}\\
&  =\ _{t}\mathbb{I}_{b_{1}^{+}}^{q+\tau}\ _{z}\mathbb{I}_{a_{1}^{+}%
}^{p+\sigma}d\left(  z,t\right)  .\nonumber
\end{align}
After using double fractional derivative operator to (\ref{eq}), we have%
\[
\ _{t}D_{b_{1}^{+}}^{q+\tau}\ _{z}D_{a_{1}^{+}}^{p+\sigma}\left(
\mathfrak{E}_{\frac{-\gamma_{1}+1}{2};-\gamma_{2}%
;-\gamma_{3};p,q,\upsilon,w_{1},w_{2};a_{1}^{+},b_{1}^{+}}^{\left(
-\gamma_{1};-\gamma_{2};-\gamma_{3}\right)  }\phi\right)  \left(  z,t\right)  =d\left(  z,t\right)
.
\]
That is%
\[
\left(  \mathbb{D}_{\frac{\gamma_{1}+1}{2};\gamma_{2};\gamma
_{3};p,q,\upsilon,w_{1},w_{2};a_{1}^{+},b_{1}^{+}}^{\left(
\gamma_{1};\gamma_{2};\gamma_{3}\right) }\phi\right)  \left(  z,t\right)  =d\left(  z,t\right) .
\]
Thus we get $\mathbb{D}_{\frac{\gamma_{1}+1}{2};\gamma_{2};\gamma
_{3};p,q,\upsilon,w_{1},w_{2};a_{1}^{+},b_{1}^{+}%
}^{\left(  \gamma_{1};\gamma_{2};\gamma_{3}\right)  }$, the left inverse of $\mathfrak{E}_{\frac{\gamma_{1}+1}{2};\gamma_{2};\gamma
_{3};p,q,\upsilon
,w_{1},w_{2};a_{1}^{+},b_{1}^{+}}^{\left(  \gamma_{1};\gamma_{2};\gamma
_{3}\right)  }$ on the function space
$L\left(  \left(  a_{1},a_{2}\right)  \times\left(  b_{1},b_{2}\right)
\right)  $.
\end{proof}

\section{Conclusion}
In this paper, our main purpose is to give a description of a family of finite biorthogonal polynomials in two variables and discover several properties. On the other hand the corresponding Mittag-Leffler functions are defined, and the constructions of modified versions of these finite biorthogonal polynomials and the corresponding Mittag-Leffler functions have given rise to be provided the semi-group property in section 3. It seems that such an analysis is novel in the literature.

\end{document}